\documentclass[10pt, reqno]{amsart}
\usepackage{graphicx}
\usepackage{xcolor}
 \usepackage{soul}
\def\constr#1^#2{\mathrel{\mathop{\kern 0pt#1}\limits^{#2}}}
\def\build#1_#2{\mathrel{\mathop{\kern 0pt#1}\limits_{#2}}}

\theoremstyle{plain}

\newtheorem{theorem}{Theorem}[section]

\newtheorem{lemma}{Lemma}[section]

\newtheorem{proposition}{Proposition}[section]
\theoremstyle{definition}

\theoremstyle{remark}
\newtheorem{remark}{\bf Remark}[section]

\numberwithin{equation}{section}

\makeatletter
\@namedef{subjclassname@2020}{%
  \textup{2020} Mathematics Subject Classification}
\makeatother

\newcommand{\nsubseteq}{\nsubseteq}

\newcommand{\R}{\mathbb{R}}
\newcommand{\N}{\mathbb{N}}

\newcommand{\ds}{\displaystyle}

\newcommand{\be}{\begin{equation}}
\newcommand{\ee}{\end{equation}}

\DeclareGraphicsExtensions{.eps} \DeclareGraphicsExtensions{.jpeg}

\begin{document}
\title[Abstract thermoelastic systems with
delay]
{Well-posedness and stability of abstract thermoelastic delayed systems}
\thispagestyle{empty}

\author{Ka\"{\i}s Ammari}
\address{LR Analysis and Control of PDEs, LR 22ES03,
Department of Mathematics, Faculty of Sciences of Mona-
stir, University of Monastir, 5019 Monastir, Tunisia}
\email{kais.ammari@fsm.rnu.tn}

\author{Makrem Salhi}
\address{LR Analysis and Control of PDEs, LR 22ES03, Department of Mathematics, Faculty of
Sciences of Sfax, University of Sfax, Tunisia}
\email{makram.salhi@ipeis.usf.tn}

\author{Farhat Shel}
\address{LR Analysis and Control of PDEs, LR 22ES03,
Department of Mathematics, Faculty of Sciences of
Monastir, University of Monastir, 5019 Monastir, Tunisia}
\email{farhat.shel@fsm.rnu.tn}

\date{}

\begin{abstract} In this paper, we consider a
stabilization problem of a generalized thermoelastic system (the so called $\alpha$-$\beta$ system) with delay in a part of the coupled system. For each case, we prove the well-posedness of the corresponding system using semigroup approach, then under some sufficient conditions we establish some results of exponential and polynomial stability of the system through a frequency-domain approach. The results are applied to concrete examples in thermoelasticity.
\end{abstract}


\maketitle

\tableofcontents

\section{Introduction}

Let $H$ be a Hilbert space equipped with an inner product
$(.,.)_{H}$ and a norm $\|\;\|_{H}$, $A:D(A)\subset H\rightarrow H$  a self-adjoint, strictly positive (unbounded) operator and $(\alpha,\beta)\in[0,1]\times[0,1]$. We consider  the following two Cauchy problems modeled, respectively, by the abstract thermoelastic systems with delay given by
 \begin{equation}\label{s0}
\left\{
\begin{array}{ll}
u''(t)+ Au(t-\tau)+aAu'(t)-A^{\beta}\theta(t)=0, & \quad t\in(0,+\infty), \\
\theta'(t)+A^{\alpha}\theta(t)+A^{\beta}u'(t)=0, & \quad t\in(0,+\infty), \\
u(0)=u_{0}, u'(0)=u_{1}, \theta(0)=\theta_{0}, &  \\
\ds A^{1/2}u(t-\tau)=\phi(t-\tau), & \quad t\in(0,\tau),
\end{array}
\right.
\end{equation}
and
\begin{equation}\label{s1}
\left\{
\begin{array}{ll}
u''(t)+Au(t)-A^{\beta}\theta(t)=0, & \quad t\in(0,+\infty), \\
\theta'(t)+\kappa A^\alpha\theta(t-\tau)+aA^\alpha\theta(t)
+A^{\beta}u'(t)=0, & \quad t\in(0,+\infty), \\
u(0)=u_{0}, u'(0)=u_{1}, \theta(0)=\theta_{0}, &  \\
\ds A^{\alpha/2} 
\theta (t-\tau) = \psi (t-\tau), & \quad t \in (0,\tau).
\end{array}
\right.
\end{equation}
where $\tau>0$ is a constant time delay,  $a>0$ and $\kappa>0$. We are interested in the well-posedness of systems (\ref{s0}) and (\ref{s1}) and the study of the asymptotic behavior of solutions $u,\theta:[0,\infty)\rightarrow H$ as $t\rightarrow\infty$, for $(\beta,\alpha)$ in the region $Q$ where $Q$ is given by $$Q:=\{(\beta,\alpha)\in[0,1]\times [0,1]\mid 2\beta-\alpha\leq 1\} \quad \text{(Figure \ref{QQ})}.$$


\begin{figure}[th]
\centering
\includegraphics[width=7cm, keepaspectratio =true]{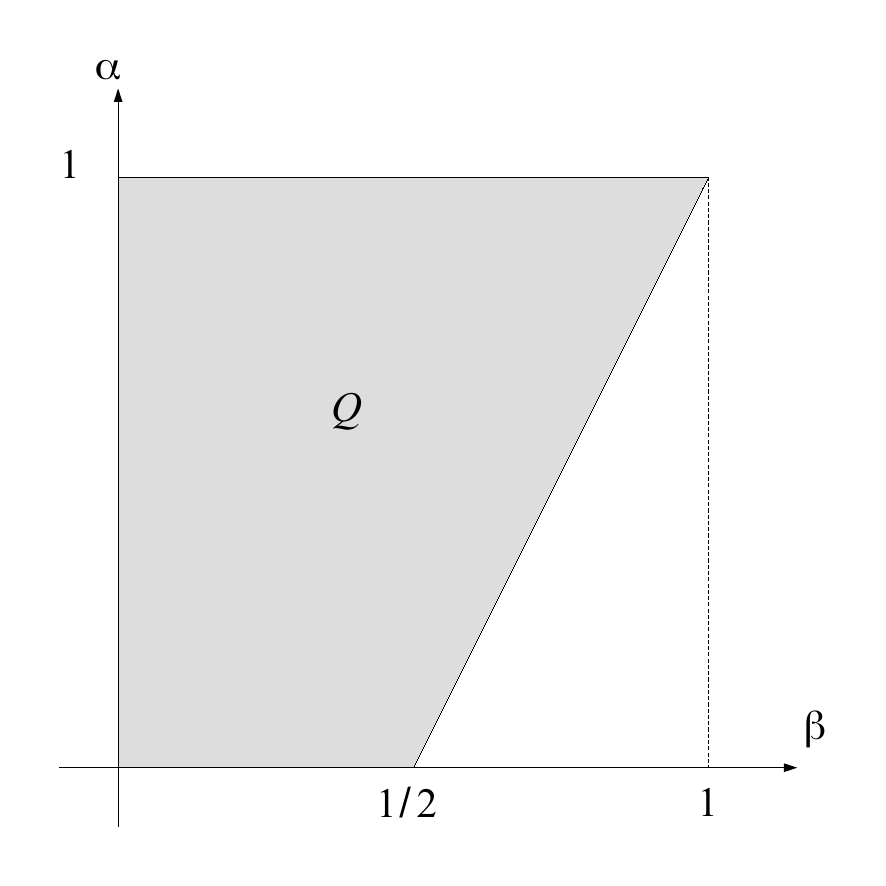}
\caption{Region $Q$}
\label{QQ}
\end{figure}


\medskip

Note that the regions of the couple $(\beta,\alpha)$ considered for the two delayed $\alpha,\beta$ systems (\ref{s0}) and (\ref{s1}) cover some concrete examples given in the literature (in the non-delayed form), such as, the thermoelastic plate equations (with $\alpha=\beta=1/2$ and $A=\Delta^2$), which is widely discussed since some decades (see e.g. \cite{Kim92, LiZh93, RiRa96,LiZh97, AvLa97, LiZh99}). The choice $H= L^2(0,L)$, with $L>0$, $A=-\frac{\partial^4}{\partial x^4}$,  $\alpha=\frac{1}{2}$ and $\beta=0$ corresponds to a one dimensional thermoelastic beam and it was proved that the corresponding system is polynomially stable of order $1$ for various boundary conditions (\cite{HaLi13, Naf23, Naf20}). 

\medskip

In recent years, several research has been dealing with delayed equations. In fact, delay effects arise in so
many applications and physical problems, see e.g. \cite{SuBi80, ADB93} and  \cite{Benhassi, Ammari, Ammari1, Ammari2, Pruss, DS, Batkai}, but may be source of some instabilities (\cite{DLP86, Dat88, Dat97, Dreher, NiPi06, Jordan, Rac12}).

\medskip

In order to ensure well-posedness or stability of a delayed system, many ideas were recently deployed: one can add a non-delay term (see \cite{Pruss} for example). A delayed system can also be stabilized using standard feedback compensating the
destabilizing delay effect
(see e.g. \cite{Ammari1, Benhassi, NiPi06, NiPi08, NPV11, NiPi18, Messaoudi}).
Furthermore, recent papers shoes that, under a particular
choice of the time delay, we may restitute the exponential stability property
(see e.g. \cite{Ammari}, \cite{Ammari2}, \cite{Gugat}).

\medskip

In \cite{Rac12}, Racke considered the  $\alpha$-$\beta$ systems with delay
\begin{equation}\label{s0'}
\left\{
\begin{array}{ll}
u''(t)+ Au(t-\tau)-A^{\beta}\theta(t)=0, & \quad t\in(0,+\infty), \\
\theta'(t)+A^{\alpha}\theta(t)+A^{\beta}u'(t)=0, & \quad t\in(0,+\infty),
\end{array}
\right.
\end{equation}
and
\begin{equation}\label{s1'}
\left\{
\begin{array}{ll}
u''(t)+Au(t)-A^{\beta}\theta(t)=0, & \quad t\in(0,+\infty), \\
\theta'(t)+\kappa A^\alpha\theta(t-\tau)
+A^{\beta}u'(t)=0, & \quad t\in(0,+\infty),
\end{array}
\right.
\end{equation}
with some initial and boundary conditions.
 He proved that under the hypothesis that the operator $A$ has a countable complete orthonormal system of eigenfunctions $(e_j)$ with corresponding eigenvalues $0<\lambda_j\rightarrow\infty$ as  $j\rightarrow \infty$, systems (\ref{s0'}) and (\ref{s1'}) are respectively unstable (not well-posed in the sense of Hadamard) at least in the regions
$$\mathcal{A}_1:=\{ (\beta,\alpha)\mid 0 \leq \beta \leq \alpha \leq 1,\;\alpha\geq \frac{1}{2},\;(\beta,\alpha)\neq (1,1)\} $$
and
$$\mathcal{A}_2:= \{ (\beta,\alpha)\mid 0 \leq \beta\leq \alpha \leq 1,\;(\beta,\alpha)\neq (1,1)\}.$$
\begin{figure}[tbp]
\centering
\begin{minipage}[t]{7cm}
\centering
\includegraphics[width=7cm, keepaspectratio =true]{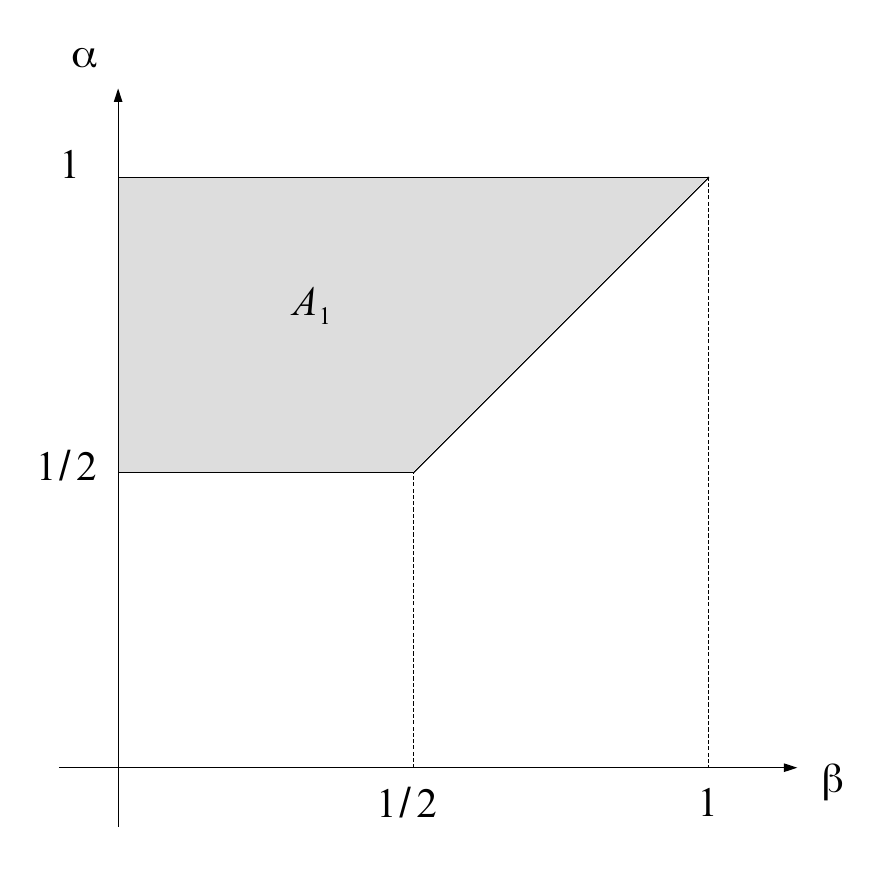}
\caption{Area of instability $A_1$ (with delay and no damping)}
\label{A1}
\hfill
\end{minipage}
\begin{minipage}[t]{7cm}
\centering
\includegraphics[width=7cm, keepaspectratio =true]{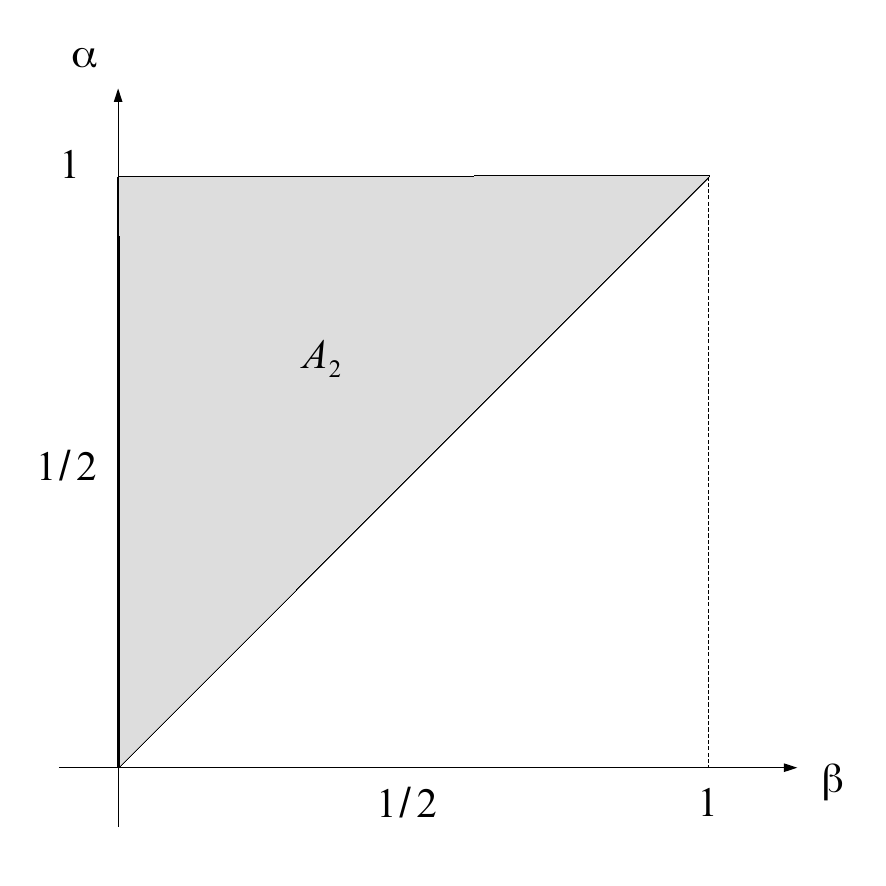}
\caption{Area of instability $A_2$ (with delay and no damping)}
\label{A2}
\end{minipage}
\end{figure}
However, it is well known that the $\alpha$-$\beta$ system without delay is well posed in the whole region $[0,1]\times[0,1]$. Moreover, in \cite{ABB99} F. A. Khodja et \textit{al.} identified the region of exponential stability to be
$$S:=\{ (\beta,\alpha)\in[0,1]\times[0,1]\mid \max\{1-2\beta, 2\beta-1\} \leq \alpha \leq 2\beta\}.$$
Afterwards Hao and Liu \cite{HaLi13} (see also \cite{ALS23, ASbook, AHbook}) give a classification of the regions of stability and instability (Figure \ref{SS}). They showed that the region of polynomial stability is given by 
$R=S_1\cup S_2$
where 
$$ 
S_1 :=\{ (\beta,\alpha)\in[0,1]\times[0,1] \mid  \max{\left(\frac{1}{2},2\beta \right)} < \alpha \}, \, S_2 := \{(\beta,\alpha)\in[0,1]\times[0,1] \mid 
\alpha < 1 - 2\beta, \alpha \leq \frac{1}{2}\}$$
and the region of instability is given by
$$S_3:=\{ (\beta,\alpha)\in[0,1]\times[0,1]\mid  0< \alpha < 2\beta-1\}.$$
Note that, the region of stability is then, $\fbox{$S\cup S_1\cup S_2=Q$}.$

\begin{figure}[th]
\centering
\includegraphics[width=7cm, keepaspectratio =true]{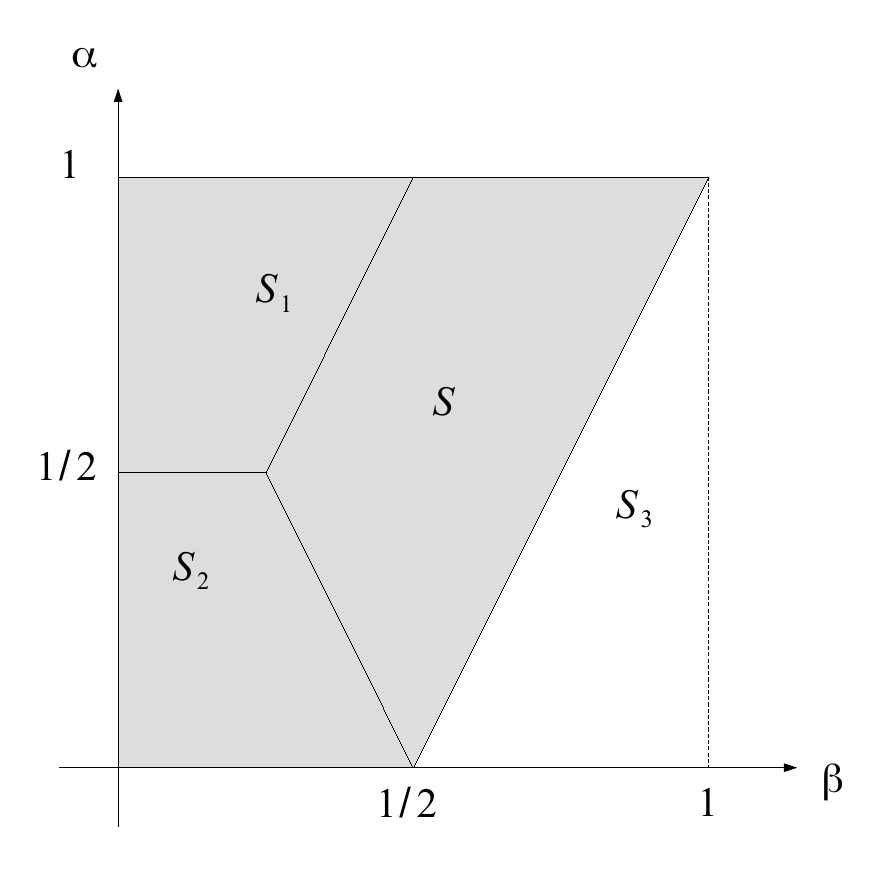}
\caption{Region of stability and region of instability (no delay \& no damping)}
\label{SS}
\end{figure}

Recently K. Ammari et \textit{al.} give more stability results taking into account the irregularity of the associated operator at the origin  \cite{ALS23}. See also  \cite{HLY15} for a complete study of the regularity of solutions of the $\alpha$-$\beta$ system, see \cite{SLR19, Naf23} and some references therein for recent works about stability of abstract thermoelastic systems.

\medskip

Besides, Racke considered in the same work \cite{Rac12} the one dimensional delayed thermoelastic systems,
\begin{equation}
\left\{ 
\begin{tabular}{l}
$u_{tt}(x,t)-\alpha u_{xx}(x,t-\tau)+\gamma \theta _x(x,t)=0,\;\;\;\;\;\;\text{ in }(0,\ell )\times (0,\infty ),$ \\
$\theta_{t}(x,t)-\kappa\theta_{xx}(x,t)+\gamma
u_{xt}(x,t)=0,\;\;\;\;\;\;\;\;\;\;\;\;\;\text{ in }(0,\ell)\times (0,\infty ),$
\end{tabular}
\right.  \label{d1.0.0}
\end{equation}
and
\begin{equation}
\left\{ 
\begin{tabular}{l}
$u_{tt}(x,t)-\alpha u_{xx}(x,t)+\gamma \theta _x(x,t)=0,\;\;\;\;\;\;\text{ in }(0,\ell )\times (0,\infty ),$ \\
$\theta_{t}(x,t)-\kappa\theta_{xx}(x,t-\tau)+\gamma
u_{xt}(x,t)=0,\;\;\;\;\;\;\;\;\;\;\;\;\;\text{ in }(0,\ell)\times (0,\infty).$
\end{tabular}
\right.  \label{d2.0.0}
\end{equation}
He proved also that theses systems are ill-posed (even if $\tau$ is small). However, we know that in the absence of delay, the classical thermoelastic system in one dimension corresponding to (\ref{d1.0.0}) or (\ref{d2.0.0}) with $\tau=0$, is well-posed and even exponentially stable (see e.g. \cite{Rac02, LiZh99}).

\medskip

In order to restitute the well-posedness character and the stability of the one dimensional thermoelastic system (\ref{d1.0.0}) Khatir and Shel \cite{KhSh21} added a Kelvin-Voight damping $-a u_{xxt}(x,t)$ to the instable system part and proved that it is not only well-posed (in an appropriate Hilbert space) but also exponentially stable for $a>a_0$ (for some positive number $a_0$. Rather, Mustapha and Kafini in\cite{MuKa13} added to the delayed equation in  (\ref{d2.0.0}) a non delayed term $-a\theta_{xx}(x,t)$ and proved a result of exponential stability under the condition $\kappa<a$.

\medskip

Even if systems (\ref{d1.0.0}) and (\ref{d2.0.0}) are not direct examples of (\ref{s0'}) and (\ref{s1'}), we try in this work, to generalize such procedure to (\ref{s0'}) and (\ref{s1'}) by adding a Kelvin-Voigt damping term $a \, Au^\prime (t)$ to the delayed equation in (\ref{s0'}) to obtain system (\ref{s0}), and a damping term $aA^{\alpha}\theta(t)$ in the second equation in (\ref{s1'}) to obtain system (\ref{s1}). We will prove in particular that, not only the systems (\ref{s0}) is well-posed (in an appropriate Hilbert space), but also, the associated semigroup is exponentially stable for $(\beta,\alpha)$ in the region $Q$. In particular the presence of the damping terms $a \, Au^\prime (t)$  provides damped systems with better stability than that acquired for $\tau=0$ at least in the regions $S_1$ and $S_2$. For system (\ref{s1}), we will prove that the presence of the damping term $a \, A^{\alpha}\theta(t)$ preserves well-posedness in  region $Q$, exponential stability in region $S$, and polynomial stability in region $S_1\cup S_2$,  acquired for $\tau=0$.

\medskip

Finally, to cover more classical thermoelastic systems, such as, the delayed systems of linear second order-thermoelasticity in one space dimension (\ref{d1.0.0}) and (\ref{d2.0.0}), we will consider in this work more general abstract thermoelastic systems with delay of the form
\begin{equation}\label{abs1}
\left\{
\begin{array}{ll}
u''(t)+ Au(t-\tau)+aAu'(t)-C\theta(t)=0, & \quad t\in(0,+\infty), \\
\theta'(t)+A^{\alpha}\theta(t)+C^{*}u'(t)=0, & \quad t\in(0,+\infty), \\
u(0)=u_{0}, u'(0)=u_{1}, \theta(0)=\theta_{0}, &  \\
\ds B^{*}u(t-\tau)=\phi(t-\tau), & \quad t\in(0,\tau),
\end{array}
\right.
\end{equation}
and 
\begin{equation}\label{abs2}
\left\{
\begin{array}{ll}
u''(t)+ Au(t)-C\theta(t)=0, & \quad t\in(0,+\infty), \\
\theta'(t)+\kappa A^\alpha \theta(t-\tau)+aA^\alpha \theta(t)+C^{*}u'(t)=0, & \quad t\in(0,+\infty), \\
u(0)=u_{0}, u'(0)=u_{1}, \theta(0)=\theta_{0}, &  \\
\ds B^{*}\theta(t-\tau)=g_0(t-\tau), & \quad t\in(0,\tau),
\end{array}
\right.
\end{equation}
where $B:D(B)\subset H\rightarrow H$, $C:D(C)\subset H\rightarrow H$ are closed densely defined operators on $H$ satisfying some properties for each system, $B^*$ is the adjoint of $B$, and  with $A = BB^*$ in (\ref{abs1}) and $A^\alpha = BB^*$ in (\ref{abs2}). The corresponding non delayed system (where $\tau=0$) was studied before in a more general form (by taking an operator $A_1$ instead of $A^\alpha$), see e.g. \cite{GRT92, RiRa95, LiLi97}) and \cite{AmBe00} and some references therein.  We will give sufficient conditions on $A^\alpha$, $B$ and $C$ which ensure the exponential (or polynomial) stability of the associated semigroup, and which can be compared to those proposed in \cite{AmBe00} for $\tau=0$. 

\medskip

The paper is organized as follows. The second and third sections deals with the well-posedness and stability
of the problems (\ref{s0}) and (\ref{s1}) respectively. We end each section by a general abstract system and applications.

\medskip

In the sequel,  $\fbox{$(\beta,\alpha)\in Q:=\{(\beta,\alpha)\in[0,1]\times [0,1]\mid 2\beta-\alpha\leq 1\}$}$.

\section{Well-posedness and stability of the problem (\ref{s0})}

\subsection{Well-posedness of the coupled system (\ref{s0})}\label{sec21} We introduce, as in \cite{Ammari}, the auxiliary variable
$$\ds z(\rho,t)=A^{1/2}u(t-\tau\rho),\quad\rho\in(0,1), t>0.$$
Then, problem (\ref{s1}) is equivalent to

\begin{equation*}\label{s2}
\ds u''(t)+A
^{1/2}z(1,t)+aAu'(t)-A^{\beta}\theta(t)=0,\quad t>0,
\end{equation*}

\begin{equation*}\label{s3}
\theta'(t)+A^{\alpha}\theta(t)+A^{\beta}u'(t)=0,\quad t>0,
\end{equation*}

\begin{equation*}\label{s4}
\tau z_{t}(\rho,t)+z_{\rho}(\rho,t)=0,\quad (\rho,t)\in(0,1)\times(0,+\infty),
\end{equation*}

\begin{equation*}
u(0)=u_{0}, u'(0)=u_{1}, \theta(0)=\theta_{0},
\end{equation*}

\begin{equation*}\label{s5}
z(\rho,0)=\phi(-\tau\rho),\quad \rho\in(0,1),
\end{equation*}

\begin{equation*}\label{s6}
\ds z(0,t)=A^{1/2}u(t),\quad t>0.
\end{equation*}

\noindent Define
$$U=(u,u',\theta,z)^{\top},$$
Then, problem (\ref{s0}) can be formulated as a first order system of the form

\begin{equation}\label{s7}
\left\{
\begin{array}{ll}
U^\prime (t)=\mathcal{A}_{\alpha,\beta}U(t), \, t > 0, & \\
\ds 
U(0)=\big(u_{0},u_{1},\theta_{0},\phi( - \tau\cdot)\big)^{\top},
\end{array}
\right.
\end{equation}
where the operator $\mathcal{A}_{\alpha,\beta}$ is defined by
$$\mathcal{A}_{\alpha,\beta}
\left(
\begin{array}{c}
u \\
v \\
\theta \\
z \\
\end{array}
\right)=
\left(
\begin{array}{c}
v \\
\ds-A^{1/2}\big(z(1)+aA^{1/2}v-A^{\beta-\frac{1}{2}}\theta\big) \\
-A^{\alpha/2}\left( A^{\alpha/2}\theta+A^{\beta-\frac{\alpha}{2}}v\right)  \\
\ds-\frac{1}{\tau}z_{\rho} \\
  \end{array}
\right),$$
with domain
$$D(\mathcal{A}_{\alpha,\beta})=\left\{
\begin{array}{c}
(u,v,\theta,z)^{\top}\in D(A^{1/2})\times D(A^{1/2})
\times D(A^{\alpha/2})\times H^{1}\big((0,1),H\big): \quad  z(0)=A^{1/2}u, \\
\ds z(1)+aA^{1/2}v-A^{\beta-\frac{1}{2}}\theta \in D(A^\frac{1}{2}), \quad\textrm{ and }\quad A^{\alpha/2}\theta+A^{\beta-\frac{\alpha}{2}}v\in D(A^{\alpha/2})
\end{array}
\right\},$$
in the Hilbert space
$$\mathcal{H}=D(A^{1/2})\times H\times H\times L^{2}\big((0,1),H\big),$$
equipped with the scalar product
$$\ds \big((u,v,\theta,z)^{\top},(u_1,v_1,\theta_1,z_1)^{\top}\big)_{\mathcal{H}}
=\big(A^{1/2}u,A^{1/2}u_1\big)_{H}
+(v,v_1)_{H}+(\theta,\theta_1)_{H}+\xi\int_{0}^{1}(z,z_1)_{H}d\rho,$$
where $\xi > 0$ is a parameter that will be fixed later on.

\begin{remark}If $(\beta,\alpha)\in [0,1/2]\times[0,1]$ and $\beta \leq \alpha$, then $A^{\beta-\frac{\alpha}{2}}v\in D(A^{\alpha/2})$, hence $\mathcal{A}_{\alpha,\beta}$ is reduced to
$$\mathcal{A}_{\alpha,\beta}
\left(
\begin{array}{c}
u \\
v \\
\theta \\
z \\
\end{array}
\right)=
\left(
\begin{array}{c}
v \\
\ds-A^{1/2}\big(z(1)+aA^{1/2}v\big)+A^{\beta}\theta \\
- A^{\alpha}\theta-A^{\beta}v  \\
\ds-\frac{1}{\tau}z_{\rho} \\
  \end{array}
\right),$$
with domain
$$D(\mathcal{A}_{\alpha,\beta})=\left\{
\begin{array}{c}
(u,v,\theta,z)^{\top}\in D(A^{1/2})\times D(A^{1/2})
\times D(A^{\alpha})\times H^{1}\big((0,1),H\big): \\
\ds z(0)=A^{1/2}u,\quad\textrm{ and }\quad z(1)+aA^{1/2}v \in D(A^\frac{1}{2})
\end{array}
\right\}.$$

\end{remark}
\begin{lemma}\rm\label{lemma21}
If $\ds\xi\geq \frac{2\tau}{a}$, then
$\mathcal{A}_{\beta,\alpha}-mId$ is dissipative in $\mathcal{H}$, where $m = \frac{1}{a} + \frac{\xi}{2\tau}$.
\end{lemma}

\begin{proof}
Take $U=(u,v,\theta,z)^{\top}\in D(\mathcal{A}_{\alpha,\beta})$.
We have
\begin{eqnarray*}
\ds\big(\mathcal{A}_{\alpha,\beta}U,U\big)_{\mathcal{H}}&=&
\big(A^{1/2}v,A^{1/2}u\big)_{H}-\big(z(1)
+aA^{1/2}v,A^{1/2}v\big)_{H}
+(A^{\beta-\frac{1}{2}}\theta,A^{\frac{1}{2}}v)_{H}-(A^{\alpha/2}\theta,A^{\alpha/2}\theta)_{H}\\&-&(A^{\frac{1}{2}}v,A^{\beta-\frac{1}{2}}\theta)_{H}
-\frac{\xi}{\tau}\int_{0}^{1}(z_{\rho},z)_{H}d\rho.
\end{eqnarray*}
Using the Young's inequality and that $\ds z(0)=A^{1/2}u$,
we find that, for every $\varepsilon>0$,
\begin{eqnarray*}
\ds Re\big(\mathcal{A}_{\alpha,\beta}U,U\big)_{\mathcal{H}}&=&
\big(A^{1/2}v,A^{1/2}u\big)_{H}-\big(z(1),A^{1/2}v\big)_{H}
-a\|A^{1/2}v\|^{2}_{H}-\|A^{\frac{\alpha}{2}}\theta\|^{2}_{H} \\
&-&\frac{\xi}{2\tau}\|z(1)\|^{2}_{H}+\frac{\xi}{2\tau}\|A^{1/2}u\|^{2}_{H} \\
&\leq&(\varepsilon-a)\|A^{1/2}v\|_{H}^{2}+\left(\frac{1}{2\varepsilon}+
\frac{\xi}{2\tau}\right)\|A^{1/2}u\|_{H}^{2}-\|A^{\frac{\alpha}{2}}\theta\|_{H}^{2}
+\left(\frac{1}{2\varepsilon}-\frac{\xi}{2\tau}\right)\|z(1)\|^{2}_{H}.
\end{eqnarray*}
Choosing $\ds\varepsilon=\frac{a}{2}$, we get
$$\ds Re\big(\mathcal{A}_{\alpha,\beta}U,U\big)_{\mathcal{H}}\leq
-\frac{a}{2}\|A^{1/2}v\|_{H}^{2}
+\left(\frac{1}{a}+\frac{\xi}{2\tau}\right)
\|A^{1/2}u\|_{H}^{2}-\|A^{\frac{\alpha}{2}}\theta\|_{H}^{2}
+\left(\frac{1}{a}-\frac{\xi}{2\tau}\right)\|z(1)\|^{2}_{H}.$$
Then, we choose $\xi$ such that $\ds\frac{1}{a}-\frac{\xi}{2\tau}\leq 0$, that
is $\ds\xi\geq\frac{2\tau}{a}$. Furthermore,
we take $\ds m=\frac{1}{a}+\frac{\xi}{2\tau}$
which leads to
\begin{equation}\label{s8}
\ds Re\big(\mathcal{A}_{\alpha,\beta}U,U\big)_{\mathcal{H}}\leq
-\frac{a}{2}\|A^{1/2}v\|_{H}^{2}
+\left(\frac{1}{a}-\frac{\xi}{2\tau}\right)
\|z(1)\|^{2}_{H}-\|A^{\frac{\alpha}{2}}\theta\|_{H}^{2}
+m\|A^{1/2}u\|_{H}^{2}.
\end{equation}
This readily shows the dissipativeness of the operator
$\mathcal{A}_{\alpha,\beta}-mId$. 
\end{proof}

In the sequel, denote by 
\begin{equation*}
\mathbb{C}_0:=\big\{\lambda\in\mathbb{C}\mid\textrm{Re}(\lambda)>0\big\},\quad \text{and}\quad \overline{\mathbb{C}}_0:=\big\{\lambda\in\mathbb{C}\mid\textrm{Re}(\lambda)\geq 0\big\}
\end{equation*}
\begin{lemma}\rm\label{lemma21"}
Assume that $a\geq \tau$. Then 
$$\overline{\mathbb{C}}_0\subset \rho(\mathcal{A}_{\alpha,\beta})$$
\end{lemma}
\begin{proof}
Since $\mathcal{A}_{\alpha,\beta}-mId$ is dissipative, it suffices to prove that  $\lambda Id-\mathcal{A}_{\alpha,\beta}$
is bijective for every $\lambda\in\overline{\mathbb{C}}_0$.

\medskip

To this end, we fix $(f,g,p,h)^{\top}\in\mathcal{H}$ and we seek
a unique solution $U=(u,v,\theta,z)^{\top}\in D(\mathcal{A}_{\alpha,\beta})$ of
$$\ds\big(\lambda Id-\mathcal{A}_{\alpha,\beta}\big)
\left(
\begin{array}{c}
u \\
v \\
\theta \\
z \\
\end{array}
\right)=
\left(
\begin{array}{c}
f \\
g \\
p \\
h \\
\end{array}
\right),$$
that is verifying
\begin{equation}\label{s9}
\left\{
\begin{array}{ll}
\lambda u-v=f, &  \\
\ds\lambda v+A^{1/2}\big(z(1)+aA^{1/2}v-A^{\beta-\frac{1}{2}}\theta\big)=g, &  \\
\lambda\theta+A^{\alpha/2}\left( A^{\alpha/2}\theta+A^{\beta-\frac{\alpha}{2}}v\right) =p, &  \\
\ds\lambda z+\frac{1}{\tau}z_\rho=h. &
\end{array}
\right.
\end{equation}
The case $\lambda=0$ is immediate. So in the rest of the proof we take $\lambda\neq 0$.

\medskip

Suppose that we have found $u$ with the appropriate regularity. Then,
\begin{equation}\label{s10}
v=\lambda u-f.
\end{equation}
To determine $z$, we set $\ds z(0)=A^{1/2}u$. Then,
by $(\ref{s9})_{4}$, we obtain
\begin{equation}\label{s11}
\ds z(\rho)=e^{-\lambda\tau\rho}A^{1/2}u+\tau
e^{-\lambda\tau\rho}\int_{0}^{\rho}h(s)e^{\lambda\tau s}ds.
\end{equation}
In particular, we have
\begin{equation}\label{s12}
z(1)=e^{-\lambda\tau}A^{1/2}u+z_{0}
\end{equation}
with
\begin{equation}\label{s13}
z_{0}=\tau e^{-\lambda\tau}\int_{0}^{1}h(s)e^{\lambda\tau s}ds\in L^{2}\big((0,1),H\big).
\end{equation}
Using (\ref{s10}) in (\ref{s9})$_2$, we see that $u$ satisfies
\begin{equation}\label{s14}
\ds\lambda^{2}(u,x)_{H}-\lambda(f,x)_{H}+(z(1)+\lambda aA^{1/2}u-
A^{\beta-1/2}\theta,A^{1/2}x)_{H}-a(A^{1/2}f,A^{1/2}x)_{H}=(g,x)_{H},
\end{equation}
for all $\ds x\in D(A^{1/2})$. 

\medskip

Substituting (\ref{s12}) in (\ref{s14}), we get
\begin{equation}\label{s15}
\ds\lambda^{2}(u,x)_{H}+(e^{-\lambda\tau}+\lambda a)
(A^{1/2}u,A^{1/2}x)_{H}-
(A^{\beta-1/2}\theta,A^{1/2}x)_{H}=(\lambda f+g,x)+
(aA^{1/2}f-z_{0},A^{1/2}x)_{H}
\end{equation}
for all $\ds x\in D(A^{1/2})$. 

\medskip

Moreover, plugging (\ref{s10}) in (\ref{s9})$_{3}$, we obtain
\begin{equation}\label{s16}
\ds\lambda(\theta,y)_{H}+(A^{\alpha/2}\theta,A^{\alpha/2}y)_{H}
+\lambda(A^{\beta-\frac{\alpha}{2}}u,A^{\alpha/2}y)_{H}=(p,y)_{H}+(A^{\beta-\frac{\alpha}{2}}f,A^{\alpha/2}y)_{H}
\end{equation}
for all $\ds y\in D(A^{\alpha/2})$. 

\medskip

Next, summing (\ref{s16}) and (\ref{s15}) multiplied by $\ds\overline{\lambda}$,
to get
\begin{equation}\label{s17}
b\big((u,\theta),(x,y)\big)=F(x,y)
\end{equation}
with
\begin{eqnarray*}
b\big((u,\theta),(x,y)\big)&=&\lambda|\lambda|^{2}(u,x)_{H}+
(\overline{\lambda}e^{-\lambda\tau}+|\lambda|^{2}a)
(A^{1/2}u,A^{1/2}x)_{H}-\overline{\lambda}(A^{\beta-1/2}\theta,A^{1/2}x)_{H}\\
&+&\lambda(\theta,y)_{H}+(A^{\frac{\alpha}{2}}\theta,A^{\frac{\alpha}{2}}y)_{H}
+\lambda(A^{1/2}u,A^{\beta-1/2}y)_{H},
\end{eqnarray*}
and
$$\ds F(x,y)=\overline{\lambda}(\lambda f+g,x)_{H}+
\overline{\lambda}(aA^{1/2}f-z_{0},A^{1/2}x)_{H}+
(p,y)_{H}+(A^{\beta-\frac{\alpha}{2}}f,A^{\alpha/2}y)_{H}.$$
The space
$$\mathcal{F}=D(A^{1/2})\times D(A^{\frac{\alpha}{2}})$$
endowed with the inner product
$$\ds\big((x,y),(x_1,y_1)\big)_{\mathcal{F}}=(A^{1/2}x,
A^{1/2}x_{1})_{H}+(x,
x_{1})_{H}+
(A^{\frac{\alpha}{2}}y,A^{\frac{\alpha}{2}}y_1)_{H}+
(y,y_1)_{H}$$
is a Hilbert space. 

\medskip

Using that $2\beta\leq \alpha+1$, we claim that the sesquilinear form $b$ is continuous on $\mathcal{F}\times\mathcal{F}$.
In fact, we have
$$\ds\big|b\big((u,\theta),(x,y)\big)\big|\leq
M\|(u,\theta)\|_{\mathcal{F}}\|(x,y)\|_{\mathcal{F}},
\quad\forall(u,\theta),(x,y)\in\mathcal{F}\times\mathcal{F},$$
for some positive constant $M$.
The anti-linear form $F$ is also clearly continuous on the Hilbert space $\mathcal{F}$. 

\medskip

Now, it is clear that for all $(x,y)\in\mathcal{F}$, we have
\begin{equation}\label{s18}
\ds Re\big(b\big((x,y),(x,y)\big)\big)=
\textrm{Re}(\lambda)|\lambda|^{2}\|x\|_{H}^{2}+
\big(\textrm{Re}(\overline{\lambda}e^{-\lambda\tau})
+|\lambda|^{2}a\big)\|A^{1/2}x\|_{H}^{2}
+\textrm{Re}(\lambda)\|y\|_{H}^{2}+\|A^{\frac{\alpha}{2}}y\|_{H}^{2}.
\end{equation}

To prove that $b$ is coercive, it suffices to prove that $\textrm{Re}(\overline{\lambda}e^{-\lambda\tau})
+|\lambda|^{2}a>0$.

First, we have
\begin{equation*}
\textrm{Re}(\overline{\lambda}e^{-\lambda\tau})
+|\lambda|^{2}a\geq |\lambda|\left( |\lambda|a-1\right)>0
\end{equation*}
for $|\lambda|>\frac{1}{a}$.

\medskip

Now, we suppose that  $|\lambda|\leq\frac{1}{a}$ and putting $\lambda=r+is,\quad(r,s)\in\R_{+}\times\R^{\ast}$. Since we assume that $a\geq \tau$, we have $|s\tau|\leq 1$ and moreover
\begin{eqnarray*}
\textrm{Re}(\overline{\lambda}e^{-\lambda\tau})
+|\lambda|^{2}a&=&(r^2+s^2)a+\big(r\cos(s\tau)-s\sin(s\tau)\big)e^{-r\tau}\\
&> &e^{-r\tau}r\cos(s\tau)+\left(a-\tau \right)s^2+r^2a\\
 &\geq &\left(a-\tau \right)s^2\geq 0.
\end{eqnarray*}

\noindent We conclude that, for all $\lambda\in\mathbb{C}$
such that $\textrm{Re}(\lambda)\geq0$,
we have
\begin{equation}\label{s23}
\textrm{Re}(\overline{\lambda}e^{-\lambda\tau})
+|\lambda|^{2}a>0
\end{equation}
for $a\geq\tau$. 

\medskip

By using (\ref{s23}) in (\ref{s18}), we deduce that $b$ is
coercive on $\mathcal{F}\times\mathcal{F}$.
By the Lax-Milgram lemma, equation (\ref{s17})
has a unique solution $(u,\theta)\in\mathcal{F}$. 

\medskip

We define $v$ by (\ref{s10}) that belongs to $D(A^{1/2})$ and $z$ by (\ref{s11}) that belongs to $H^1(0,1;H)$

\medskip

 Now, if we take $x=0$ in (\ref{s17}), we get
$$(A^{\frac{\alpha}{2}}\theta-A^{\beta-\frac{\alpha}{2}} v,A^{\frac{\alpha}{2}}y)_{H}=
(p-\lambda\theta,y)_{H},\quad\forall y\in D(A^{\frac{\alpha}{2}}).$$
It follows that
$$A^{\frac{\alpha}{2}}\theta-A^{\beta-\frac{\alpha}{2}} v\in D(A^{\alpha/2})\quad\textrm{ and }\quad
A^{\alpha/2}\left(A^{\frac{\alpha}{2}}\theta-A^{\beta-\frac{\alpha}{2}} v \right) +\lambda\theta=p.$$
\noindent 
Next, taking $y=0$ in (\ref{s17}) divided by $\overline{\lambda}$ leads to
$$(aA^{1/2}v+z(1)-A^{\beta-\frac{1}{2}}\theta,A^{1/2}x)_{H}=(g-\lambda v,x)_{H},
\quad\forall x\in D(A^{1/2}).$$
Thus, we get
$$\ds aA^{1/2}v+z(1)-A^{\beta-\frac{1}{2}}\theta\in D(A^{1/2})\quad\textrm{ and }\quad
A^{1/2}\big(aA^{1/2}v+z(1)-A^{\beta-\frac{1}{2}}\theta\big)+\lambda v=g.$$
We conclude that $\lambda Id-\mathcal{A}_{\alpha,\beta}$ is bijective for all $\lambda\in\mathbb{C}$
such that $\textrm{Re}(\lambda)\geq0$ (in particular, $\mathcal{A}_{\alpha,\beta}-mI$ is m-dissipative).

\end{proof}

Since the operator $\mathcal{A}_{\alpha,\beta}-mI$ is m-dissipative, it generates a $\mathcal{C}_0$-semigroup of contractions 
on $\mathcal{H}$. Then we have the following result:

\begin{proposition}
For $(\beta,\alpha)\in Q$, $a\geq\tau$ and
$\ds\xi\geq\frac{2\tau}{a}$, the system (\ref{s0}) is well-posed.
More precisely, the operator $\mathcal{A}_{\alpha,\beta}$ generates a
$\mathcal{C}_0$-semigroup on $\mathcal{H}$.
Hence, for every
$(u_0,v_0,\theta_0,z_0)\in \mathcal{H}$, there exists a unique solution
$(u,v,\theta,z)\in\mathcal{C}\big((0,+\infty),\mathcal{H}\big)$ of (\ref{s7}).
Moreover, if
$(u_0,v_0,\theta_0,z_0)\in D(\mathcal{A}_{\alpha,\beta})$,
then we have $(u,u^\prime,\theta,z)\in\mathcal{C}\big((0,+\infty),
D(\mathcal{A}_{\alpha,\beta})\big)\cap
\mathcal{C}^{1}\big((0,+\infty),
\mathcal{H}\big)$.
\end{proposition}

\subsection{Exponential stability of the delayed coupled system (\ref{s0})}

 First, we recall the following frequency domain result for uniform
stability from \cite{Zwart}, Theorem. 8.1.4, of a
$\mathcal{C}_0$-semigroup $e^{t\mathcal{L}}$ on a Hilbert space:
\begin{lemma}\rm\label{lemma exp.stab}
A $\mathcal{C}_0$ semigroup $e^{t\mathcal{L}}$ on a Hilbert space $G$ satisfies
$$\|e^{t\mathcal{L}}\|_{\mathcal{L}(G)}\leq Ce^{-wt}$$
for some constants $C>0$ and $w>0$ if and only if

\begin{equation}\label{s24}
\overline{\mathbb{C}}_0
\subset\rho(\mathcal{L})
\end{equation}

and

\begin{equation}\label{s25}
\underset{\textrm{Re}(\lambda)\geq 0}\sup\|(\lambda I-\mathcal{L})^{-1}\|_{\mathcal{L}(G)}<\infty,
\end{equation}
where $\rho(\mathcal{L})$ denotes the resolvent set of the operator $\mathcal{L}$.
\end{lemma}

\noindent Now, we can state our main result.

\begin{theorem}\rm\label{theorem21}
For $(\beta,\alpha)\in Q$, $a\geq\tau$ and
$\ds\xi\geq\frac{2\tau}{a}$, the system (\ref{s7}) is exponentially
stable.
\end{theorem}

\begin{proof}

In view of Lemma \ref{lemma21"} the condition (\ref{s24}) is satisfied.  Moreover,

\medskip

 Now, we suppose that condition (\ref{s25}) is false. Then (using again Lemma \ref{lemma21"},
there exists a sequence of complex numbers $\lambda_n$ such that
$\textrm{Re}(\lambda_n)\geq0$, $|\lambda_n|\rightarrow+\infty$ and a sequence
of vectors $U_n=(u_n,v_n,\theta_n,z_n)^t\in D(\mathcal{A}_{\alpha,\beta})$
with
\begin{equation}\label{s26}
\ds\|A^{1/2}u_n\|^{2}_{H}+\|v_n\|^{2}_{H}+\|\theta_{n}\|_{H}^{2}
+\xi\int_{0}^{1}\|z_n(\rho)\|_{H}^{2}=1,\quad\forall n \geq 1,
\end{equation}
and such that

\begin{equation}\label{s27}
\underset{n\rightarrow+\infty}\lim\big\|(\lambda_n I-\mathcal{A}_{\alpha,\beta})U_n\big\|_{H}=0,
\end{equation}

i.e.,

\begin{equation}\label{s28}
\ds\lambda_nA^{1/2}u_n-A^{1/2}v_n\equiv f_n\rightarrow0 \textrm{ in } H,
\end{equation}

\begin{equation}\label{s29}
\ds\lambda_nv_n+A^{1/2}\big(z_n(1)+aA^{1/2}v_n
-A^{\beta-\frac{1}{2}}\theta_n\big)\rightarrow0 \textrm{ in } H,
\end{equation}

\begin{equation}\label{s30}
\ds\lambda_n\theta_n+A^{\alpha/2}\big(A^{\alpha/2}\theta_n+A^{\beta-\frac{\alpha}{2}}v_n\big)\rightarrow0 \textrm{ in } H,
\end{equation}

\begin{equation}\label{s31}
\ds\lambda_nz_n+\frac{1}{\tau}\partial_{\rho}z_n
\equiv h_n\rightarrow0 \textrm{ in }L^2\big((0,1)H\big).
\end{equation}
Since $\bigg|\big((\lambda_n I-\mathcal{A}_{\alpha,\beta})U_n,U_n\big)_{\mathcal{H}}\bigg|
\geq\textrm{Re}(\lambda _n)-\textrm{Re}(\mathcal{A}_{\alpha,\beta}U_n,U_n)_{\mathcal{H}}$,
then, using (\ref{s8}) and (\ref{s28}), we obtain that
\begin{eqnarray*}
\ds\bigg|\big((\lambda_n I-\mathcal{A}_{\alpha,\beta})U_n,U_n\big)_{\mathcal{H}}\bigg|&\geq&
\textrm{Re}(\lambda_n)+\frac{a}{2}\|A^{1/2}v_n\|_{H}^{2}+
\left(\frac{\xi}{2\tau}-\frac{1}{a}\right)\|z_n(1)\|_{H}^{2}+
\|A^{\frac{\alpha}{2}}\theta_n\|_{H}^{2}-m\|A^{1/2}u_n\|_{H}^{2} \\
&\geq&\textrm{Re}(\lambda_n)+
\left(\frac{a}{2}-\frac{2m}{|\lambda_n|}\right)\|A^{1/2}v_n\|_{H}^{2}
-\frac{2m}{|\lambda_n|}\|f_n\|_{H}^{2}+\left(\frac{\xi}{2\tau}
-\frac{1}{a}\right)\|z_n(1)\|_{H}^{2}\\
&&+\|A^{\frac{\alpha}{2}}\theta_n\|_{H}^{2}.
\end{eqnarray*}
Note that there exists $n_0\in\N$ such that
$\ds\frac{a}{2}-\frac{2m}{|\lambda_n|}\geq\frac{a}{4}$ for all
$n\geq n_0$. Thus, for $n\geq n_0$, we get
\begin{eqnarray}
\bigg|\big((\lambda_n I-\mathcal{A}_{\alpha,\beta})U_n,U_n\big)_{\mathcal{H}}\bigg| &\geq &
\textrm{Re}(\lambda_n)+\frac{a}{4}\|A^{1/2}v_n\|_{H}^{2}
-\frac{2m}{|\lambda_n|}\|f_n\|_{H}^{2} \notag\\
&+&\left(\frac{\xi}{2\tau}
-\frac{1}{a}\right)\|z_n(1)\|_{H}^{2}+\|A^{\frac{\alpha}{2}}\theta_n\|_{H}^{2},\label{s31'}
\end{eqnarray}
which further leads to

\begin{equation}\label{s32}
\ds A^{1/2}v_n\underset{n\rightarrow+\infty}\rightarrow0\;\;\quad\textrm{ and }\quad
A^{\frac{\alpha}{2}}\theta_n\underset{n\rightarrow+\infty}\rightarrow0\textrm{ in }H,
\end{equation}
(we have used the estimate (\ref{s27}) in (\ref{s31'})). Then, it follows
\begin{equation}\label{s33}
v_n\underset{n\rightarrow+\infty}\rightarrow0\quad\textrm{ and }\quad
\theta_n\underset{n\rightarrow+\infty}\rightarrow0\textrm{ in }H.
\end{equation}
Furthermore,
\begin{equation}\label{s34}
\ds A^{1/2}u_n\underset{n\rightarrow+\infty}\rightarrow0\textrm{ in }H,
\end{equation}
due to (\ref{s28}). 

\medskip

By integration of the identity (\ref{s31}), we obtain
\begin{equation}\label{s35}
\ds z_n(\rho)=A^{1/2}u_n e^{-\lambda_n\rho}+
\tau\int_{0}^{\rho}e^{-\tau\lambda_n(\rho-s)}h_n(s)ds.
\end{equation}

Combining (\ref{s34}) and (\ref{s35}), we get
\begin{equation}\label{s36}
\ds\|z_n\|_{L^{2}\big((0,1),H\big)}^{2}\leq2\|A^{1/2}u_n\|_{H}^{2}
+2\tau^{2}\|h_n\|_{L^{2}\big((0,1),H\big)}^{2}\ds\underset{n\rightarrow+\infty}\rightarrow0.
\end{equation}
Finally, we have shown that $\|U_n\|_{\mathcal{H}}\underset{n\rightarrow+\infty}\rightarrow0$
which clearly contradicts (\ref{s26}). Thus, (\ref{s25}) holds and the proof of
Theorem \ref{theorem21}
 is then finished.
\end{proof}
\begin{remark}
Adding the Kelvin-Voigt damping $aAu^\prime(t)$ to $ Au(t-\tau)$ not only restores the exponential stability in the region $S$, but  improves the  stability in the $S_1\cup S_2$ which become exponential.
\end{remark}

\subsection{Some related systems}
We consider the following system
 \begin{equation}\label{s0'"}
\left\{
\begin{array}{ll}
u''(t)+ Au(t-\tau)+aAu'(t)-C\theta(t)=0, & \quad t\in(0,+\infty), \\
\theta'(t)+A^{\alpha}\theta(t)+C^{*}u'(t)=0, & \quad t\in(0,+\infty), \\
u(0)=u_{0}, u'(0)=u_{1}, \theta(0)=\theta_{0}, &  \\
\ds B^{*}u(t-\tau)=\phi(t-\tau), & \quad t\in(0,\tau),
\end{array}
\right.
\end{equation}
where $\tau>0$ is a constant time delay and $a>0$, $B:D(B) \subset H\rightarrow H$ and $C:D(C) \subset H\rightarrow H$ are closed densely defined linear operators with $B^*$ and $C^*$ are the adjoints of $B$ and $C$. 
 We suppose that $A=BB^*$. Formally, system (\ref{s0'"}) can be seen as a generalization of the delayed  $\alpha$-$\beta$ system (\ref{s0}). 
 
 We suppose also that 
 \be
\label{coer}
\exists \, c_1 > 0 \; \hbox{ s.t. } \; \|B^{*}v\|_H\geq c_1\|v\|_{H},\;\;\forall\
v\in D(B^{*}),
\ee
 \begin{equation*}
 D(A^{\alpha})\subset D(C), \quad D(B^{*})\subset D(C^*)
 \end{equation*}
 and
\be
\label{comp}
\exists \, c_2 > 0 \; \hbox{ s.t. } \; \|C^{*}v\|_H\leq c_2\|B^*v\|_{H},\;\;\forall\
v\in D(B^{*}).
\ee  

\subsubsection{Well-posedness}
Here we take
$$\ds z(\rho,t):=B^{*}u(t-\tau\rho),\quad\rho\in(0,1), t>0.$$
Then, problem (\ref{s0'"}) is equivalent to

\begin{equation}\label{s011}
\left\{
\begin{array}{ll}
u''(t)+Bz(1,t)+aBB^*u'(t)-C\theta(t)=0, &  t>0, \\
\theta'(t)+A^{\alpha}\theta(t)+C^{*}u'(t)=0, & t>0, \\
\tau z_{t}(t,\rho)+z_{\rho}(\rho,t)=0, & (\rho,t)\in(0,1)\times(0,+\infty), \\
u(0)=u_{0}, u'(0)=u_{1}, \theta(0)=\theta_{0}, & \\
z(\rho,0)=\phi(-\tau\rho),& \rho\in(0,1),\\
z(0,t)=B^*u(t),&\quad t>0.
\end{array}
\right.
\end{equation}

Define
$U=(u,u',\theta,z)^{\top}$,
then problem (\ref{s011}) can be formulated as a first order system of the form

\begin{equation}\label{s77}
\left\{
\begin{array}{ll}
U'=\mathcal{A}U, & \\
\ds U(0)=\big(u_{0},u_{1},\theta_{0},\phi(-\tau.)\big)^{\top},
\end{array}
\right.
\end{equation}
where the operator $\mathcal{A}$ is defined by
$$\mathcal{A}
\left(
\begin{array}{c}
u \\
v \\
\theta \\
z \\
\end{array}
\right)=
\left(
\begin{array}{c}
v \\
\ds-B\big(z(1)+aB^*v\big)+C\theta \\
-A^{\alpha}\theta-C^{*}v \\
\ds-\frac{1}{\tau}z_{\rho} \\
  \end{array}
\right),$$
with domain
$$D(\mathcal{A})=\left\{
\begin{array}{c}
(u,v,\theta,z)^{\top}\in D(B^*)\times D(B^*)
\times D(A^{\alpha})\times H^{1}\big((0,1),H\big): \\
\ds z(0)=B^*u\quad\textrm{ and }\quad aB^*v+z(1)\in D(B)
\end{array}
\right\},$$
in the Hilbert space
$$\mathcal{H}=D(B^*)\times H\times H\times L^{2}\big((0,1),H\big),$$
equipped with the scalar product
$$\big((u,v,\theta,z)^\top,(u_1,v_1,\theta_1,z_1)^\top\big)_{\mathcal{H}}
=\big(B^*u,B^*u_1\big)_{H}
+(v,v_1)_{H}+(\theta,\theta_1)_{H}+\xi\int_{0}^{1}(z,z_1)_{H}d\rho,$$
where $\xi>0$ is a positive constant.

\medskip

 As in the previous case,  for $\ds\xi\geq\frac{2\tau}{a}$, $\mathcal{A}-mId$ is dissipative in $\mathcal{H}$, where $m = \frac{1}{a} + \frac{\xi}{2\tau}$. Moreover, we have the following lemma:

\begin{lemma}\rm\label{lemma23'}
Assume that $a \geq \tau$.
Then
\begin{equation}\label{s49'}
\overline{\mathbb{C}}_0
\subset\rho\big(\mathcal{A}\big).
\end{equation}
\end{lemma}

\begin{proof}
First, it is easy to show that $0\in\rho(\mathcal{A})$. Now let $\lambda \in \overline{\mathbb{C}}_0$ such that $\lambda\neq 0$.  
We will prove that  $\lambda Id-\mathcal{A}$
is bijective.
We fix $(f,g,p,h)\in\mathcal{H}$ and we solve the equation
\begin{equation}\label{bij}
\big(\lambda Id-\mathcal{A}\big)
\left(
\begin{array}{c}
u \\
v \\
\theta \\
z \\
\end{array}
\right)=
\left(
\begin{array}{c}
f \\
g \\
p \\
h \\
\end{array}
\right)
\end{equation}
with $U=(u,v,\theta,z)^{\top}\in D(\mathcal{A})$.

\medskip

Equation (\ref{bij}) is written explicitly as
\begin{equation}\label{s9'}
\left\{
\begin{array}{ll}
\lambda u-v=f, &  \\
\ds\lambda v+B\big(z(1)+aB^*v\big)-C\theta=g, &  \\
\lambda\theta+A^{\alpha}\theta+C^{*}v=p, &  \\
\ds\lambda z+\frac{1}{\tau}z_\rho=h. &
\end{array}
\right.
\end{equation}
By taking in mind that $z(0)=B^*u$, equation (\ref{s9'})$_{4}$ can be solved as follows
\begin{equation}\label{s11'}
\ds z(\rho)=e^{-\lambda\tau\rho}B^*u+\tau
e^{-\lambda\tau\rho}\int_{0}^{\rho}h(s)e^{\lambda\tau s}ds.
\end{equation}
In particular, we have
\begin{equation}\label{s12'}
z(1)=e^{-\lambda\tau}B^*u+z_{0}
\end{equation}
with
\begin{equation*}
z_{0}=\tau e^{-\lambda\tau}\int_{0}^{1}h(s)e^{\lambda\tau s}ds\in L^{2}\big((0,1),H\big).
\end{equation*}
Multiplying (\ref{s9'})$_2$ by $\lambda x$,  $\ds x\in D(B^*)$, and (\ref{s9'})$_3$ by $y\in D(A^\alpha)$ respectively, then summing the obtained results, we get, using (\ref{s9'})$_1$ and (\ref{s12'}),   

\begin{equation}\label{s17'}
b\big((u,\theta),(x,y)\big)=F(x,y)
\end{equation}
with
\begin{eqnarray*}
b\big((u,\theta),(x,y)\big)&=&\lambda|\lambda|^{2}(u,x)_{H}+
\left(\overline{\lambda}e^{-\lambda\tau}+|\lambda|^{2}a
\right)
(B^*u,B^*x)_{H}-\overline{\lambda}(\theta,C^*x)_{H}\\
&+&\lambda(\theta,y)_{H}+(A^{\frac{\alpha}{2}}\theta,A^{\frac{\alpha}{2}}y)_{H}
+\lambda(C^{*}u,y)_{H},
\end{eqnarray*}
and
$$\ds F(x,y)=\overline{\lambda}(\lambda f+g,x)_{H}+
\overline{\lambda}(aB^*f-z_{0},B^*x)_{H}+
(p+C^{*}f,y)_{H}.$$
The space
$$\mathcal{F}=D(B^*)\times D(A^{\frac{\alpha}{2}})$$
endowed with the inner product
$$\ds\big((x,y),(x_1,y_1)\big)_{\mathcal{F}}=(B^*x,
B^*x_{1})_{H}+(x,x_1)_{H}+
(A^{\frac{\alpha}{2}}y,A^{\frac{\alpha}{2}}y_1)_{H}+(y,y_1)_{H}$$
is a Hilbert space. 

\medskip

Using (\ref{comp}), we have that the sesquilinear form $b$ is continuous on $\mathcal{F}\times \mathcal{F}$. Moreover, the anti-linear form $F$ is continuous on $\mathcal{F}$. 

\medskip

 Now, for all $(x,y)\in\mathcal{F}$, we have
\begin{equation}\label{s18'}
\ds Re\bigg(b\big((x,y),(x,y)\big)\bigg)=
\textrm{Re}(\lambda)|\lambda|^{2}\|x\|_{H}^{2}+
\big(\textrm{Re}(\overline{\lambda}e^{-\lambda\tau})
+|\lambda|^{2}a\big)\|B^*x\|_{H}^{2}
+\textrm{Re}(\lambda)\|y\|_{H}^{2}+\|A^{\frac{\alpha}{2}}y\|_{H}^{2}.
\end{equation}
As in the proof of Lemma \ref{lemma21}, we have
$\textrm{Re}(\overline{\lambda}e^{-\lambda\tau})
+|\lambda|^{2}a>0$
for $a\geq\tau$. 
Then $b$ is
coercive on $\mathcal{F}\times\mathcal{F}$.
By the Lax-Milgram lemma, equation (\ref{s17'})
has a unique solution $(u,\theta)\in\mathcal{F}$. 

\medskip

We define $v$ by (\ref{s9'})$_1$ that belongs to $D(B^*)$ and $z$ by (\ref{s11'}) that belongs to $H^1(0,1;H).$

\medskip

 Now, if we consider $(0,y)\in\mathcal{F}$ in (\ref{s17'}), we get
$$(A^{\frac{\alpha}{2}}\theta,A^{\frac{\alpha}{2}}y)_{H}=
(p-\lambda\theta-C^* v,y)_{H},\quad\forall y\in D(A^{\frac{\alpha}{2}}).$$
It follows that
$$\theta\in D(A^{\alpha})\quad\textrm{ and }\quad
A^{\alpha}\theta+\lambda\theta+C^{*}v=p.$$
\noindent Next, Taking $(x,0)\in\mathcal{F}$ in (\ref{s17'}) leads to
$$(aB^*v+z(1),B^*x)_{H}=(g+C\theta-\lambda v,x)_{H},
\quad\forall x\in D(B^*).$$
Thus, we get
$$\ds aB^*v+z(1)\in D(B)\quad\textrm{ and }\quad
B\big(aB^*v+z(1)\big)+\lambda v-C\theta=g.$$

We conclude that equation (\ref{bij}) has a unique solution $U\in D(\mathcal{A})$. Hence $\lambda Id-\mathcal{A}$ is bijective. Further, by the closedness of $\mathcal{A}$ (because $0\in\rho(\mathcal{A})$), we deduce that $\lambda\in\rho(\mathcal{A})$ for every $\lambda\in \overline{\mathbb{C}}_0-\{0\}$.
\end{proof}
 In particular, we have the following result.

\begin{proposition}
For $a\geq\tau$ and
$\ds\xi\geq\frac{2\tau}{a}$, the system (\ref{s011}) is well-posed.
More precisely, the operator $\mathcal{A}$ generates a
$\mathcal{C}_0$-semigroup on $\mathcal{H}$.
Hence, for every
$(u_0,v_0,\theta_0,z_0)\in \mathcal{H}$, there exists a unique solution
$(u,v,\theta,z)\in\mathcal{C}\big((0,+\infty),\mathcal{H}\big)$ of (\ref{s77}).
Moreover, if
$(u_0,v_0,\theta_0,z_0)\in D(\mathcal{A})$,
then we have $(u,u^\prime,\theta,z)\in\mathcal{C}\big((0,+\infty),
D(\mathcal{A})\big)\cap
\mathcal{C}^{1}\big((0,+\infty),\mathcal{H}\big)$.
\end{proposition}

\subsubsection{Exponential stability of the delay coupled system (\ref{s77})}  

\begin{theorem}
    \rm\label{theorem21'}
For $a\geq\tau$ and
$\ds\xi\geq\frac{2\tau}{a}$, the system (\ref{s77}) is exponentially
stable.
\end{theorem}

\begin{proof}
Since $\overline{\mathbb{C}}_0\subset\rho(\mathcal{L})$ it suffices to prove that
\begin{equation*}
\underset{\textrm{Re}(\lambda) \geq 0}\sup\|(\lambda I-\mathcal{L})^{-1}\|_{\mathcal{L}(\mathcal{H})}<\infty.
\end{equation*}
The rest of the proof is the same as in the proof of Theorem \ref{theorem21}.
\end{proof}
\subsection{Applications}
\subsubsection{Thermoelastic plate with delay}\label{sec1}

Taking $\alpha=\beta=\frac{1}{2}$, $H=L^2(\Omega)$ where $\Omega$ is a smooth open bounded domain in $\mathbb{R}^n$,  and consider  

\begin{equation}\label{exp11}
\left\{
\begin{array}{ll}
u_{tt}(x,t)+\Delta^2u(x,t-\tau)+a \Delta^2u_{t}(x,t)+\Delta \theta(x,t)=0, &
\quad (x,t)\in\Omega\times(0,+\infty), \\
\theta_{t}(x,t)-\Delta\theta(x,t)-\Delta u_{t}(x,t)=0, &
\quad(x,t)\in\Omega\times(0,+\infty), \\
u(x,t)=\Delta u(x,t)=0, & \quad (x,t)\in\partial \Omega\times(0,+\infty), \\
u(x,0)=u^0(x),u_t(x,0)=u^1(x), & \quad t\in(0,\tau), \\
\theta(x,t)=0, & \quad (x,t)\in\partial \Omega\times(0,+\infty), \\
\theta(x,0)=\theta^0(x), &  \\
-\Delta u(x,t)=f_0(x,t), &\quad-\tau\leq t<0, x\in\Omega,
\end{array}
\right.
\end{equation}
where $\tau$ and $a$ are real positive constants.

\medskip

Here, $A^{1/2}=-\Delta$, with domain $D(A^{1/2})=H^1_0(\Omega)\cap H^2(\Omega)$, and $A=-\Delta^2$, with domain $D(A)=\{u\in H^1_0(\Omega)\cap H^2(\Omega),\,\Delta u\in H^1_0(\Omega)\cap H^2(\Omega) \}.$

\medskip

The operator $\mathcal{A}_{\frac{1}{2},\frac{1}{2}}$ is given by
\begin{equation*}
\ds\mathcal{A}_{\frac{1}{2},\frac{1}{2}}
\left(
\begin{array}{c}
u \\
v \\
\theta \\
z \\
\end{array}
\right)=
\left(
\begin{array}{c}
v \\
\ds-\Delta\big(-z(\cdot,1)+a\Delta v\big)-\Delta \theta \\
\Delta\theta+\Delta v \\
\ds-\frac{1}{\tau}z_{\rho} \\
  \end{array}
\right),
\end{equation*}
with domain
\begin{equation*}
D(\mathcal{A}_{\frac{1}{2},\frac{1}{2}})=\left\{
\begin{array}{c}
(u,v,\theta,z)^{\top}\in  \left( H_{0}^{1}\big(\Omega\big)\cap H^{2}\big(\Omega\big)\right)^3 \times H^{1}\big((0,1),L^{2}(\Omega)\big): \\
\ds \quad z(\cdot,0)=-\Delta u\quad\textrm{ and }\quad z(\cdot,1)-a\Delta v\in H_{0}^{1}\big(\Omega\big)\cap H^{2}\big(\Omega\big)
\end{array}
\right\},
\end{equation*}
in the Hilbert space
\begin{equation*}
\mathcal{H}=\left( H_{0}^{1}\big(\Omega\big)\cap H^{2}\big(\Omega\big)\right)\times
L^{2}\big(\Omega\big)\times L^{2}\big(\Omega\big)\times
L^{2}\big(\Omega\times (0,1)\big).
\end{equation*}



 Applying Theorem \ref{theorem21} with
$\ds (\beta,\alpha)=\left(\frac{1}{2},\frac{1}{2}\right)\in Q$,
one has the following exponential stability result.

\begin{theorem}
If $\tau\leq a$, the system (\ref{exp11}) is exponentially stable,
namely for $\ds\xi\geq\frac{2\tau}{a}$, the energy
$$\ds E(t)=\frac{1}{2}\left(\int_{\Omega}
\big(|u_{t}(x,t)|^{2}+|\Delta u(x,t)|^{2}+|\theta(x,t)|^{2}\big)dx
+\xi\int_{\Omega}\int_{0}^{1}|\Delta u(x,t-\tau\rho)|^{2}d\rho dx\right),$$
satisfies
$$E(t)\leq Ce^{-wt}E(0),\quad\forall t \geq 0, \, (u^0,u^1,\theta^0,f_0) \in D(\mathcal{A}_{\frac{1}{2},\frac{1}{2}}) $$
for some positive constants $C$ and $w$.
\end{theorem}

\subsubsection{Thermoelastic string system with delay}
Consider the system (\ref{s0'"}) with $H=L^2(0,L)$, $A=-\frac{\partial^2}{\partial x^2}:D(A)=H_0^1(0,L)\cap H^2(0,L)\rightarrow L^2(0,L)$, $\alpha=1$, $B=C^*=-\frac{\partial}{\partial x}:D(B)=H^1(0,L)\rightarrow L^2(0,L)$, $B^*=C=\frac{\partial}{\partial x}:D(B^*)=H_0^1(0,L)\rightarrow L^2(0,L)$. We have that $A=BB^*$, $B^*$ satisfies (\ref{coer}), $D(A)\subset D(C)$, $D(B^*)\subset D(C^*)$ and $\|C^*v\|\leq \|B^*v\|, \quad \forall\,v\in D(B^*)$. We thus find the system considered in \cite{KhSh21} with a small difference at the boundary conditions
 of $\theta$ (by considering here Dirichlet conditions instead of Neumann conditions):
 
\begin{equation}\label{s84""}
\left\{
\begin{array}{ll}
u_{tt}(t,x)-u_{xx}(x,t)-au_{xxt}(x,t-\tau)+\theta_x(x,t)=0, &
\quad (x,t)\in(0,L)\times(0,+\infty), \\
\theta_{t}(x,t)-\theta_{xx}(x,t)+u_{xt}(x,t)=0, &
\quad(x,t)\in(0,L)\times(0,+\infty), \\
u(0,t)=u(L,t)=0, & \quad t\in(0,+\infty), \\
u(x,0)=u^0(x),u_t(x,0)=u^1(x), & \quad t\in(0,\tau), \\
\theta(0,t)=\theta(L,t)=0, & \quad t\in(0,+\infty), \\
\theta(x,0)=\theta^0(x), &  \\
u_x(x,t)=f_0(x,t), &\quad-\tau\leq t<0, x\in (0,L).
\end{array}
\right.
\end{equation}

Then, the operator $\mathcal{A}$ is simply written as follows
\begin{equation*}
\ds\mathcal{A}
\left(
\begin{array}{c}
u \\
v \\
\theta \\
z \\
\end{array}
\right)=
\left(
\begin{array}{c}
v \\
\ds \left( z_{x}(\cdot,1)+av_x\right)_x-\theta_x \\
\theta_{xx}-v_x \\
\ds-\frac{1}{\tau}z_{\rho} \\
  \end{array}
\right),
\end{equation*}
with domain
\begin{equation*}
D(\mathcal{A})=\left\{
\begin{array}{c}
(u,v,\theta,z)^{\top}\in  H_{0}^{1}\big(0,L\big) \times H_{0}^{1}\big(0,L\big)
\times  \left(H_{0}^{1}\big(0,L\big)\cap H^{2}\big(0,L\big)\right) \times H^{1}\big((0,1),L^2(0,L)\big): \\
\ds \quad z(\cdot,0)=u_x\quad\textrm{ and }\quad z(\cdot,1)+av_x\in H^{1}\big(0,L\big)
\end{array}
\right\},
\end{equation*}
in the Hilbert space
\begin{equation*}
\mathcal{H}=H_{0}^{1}\big(0,L\big)\times
L^{2}\big(0,L\big)\times L^{2}\big(0,L\big)\times
L^{2}\big((0,1), L^2(0,L)\big).
\end{equation*}



 Applying Theorem \ref{theorem21'}
one has the following exponential stability result.

\begin{theorem}
If $\tau\leq a$, the system (\ref{s84""}) is exponentially stable.
\end{theorem}

\section{Well-posedness and stability of the problem (\ref{s1})}

\subsection{Well-posedness of the delayed coupled system (\ref{s1})}

In this subsection, we suppose that $(\beta,\alpha)\in Q$.

 We introduce, as in section \ref{sec21} or in \cite{Ammari}, the auxiliary variable
$$\ds z(\rho,t)=A^{\frac{\alpha}{2}}\theta(t-\tau\rho),\quad\rho\in(0,1), t>0.$$
Then, problem (\ref{s1}) is equivalent to

\begin{equation*}\label{s41}
\ds u''(t)+Au(t)+aAu'(t)-A^{\beta}\theta(t)=0,\quad t>0,
\end{equation*}

\begin{equation*}\label{s42}
\theta'(t)+\kappa A^{\frac{\alpha}{2}}\big(z(1)+\frac{a}{\kappa}
A^{\frac{\alpha}{2}}\theta(t)\big)+A^{\beta}u'(t)=0,\quad t>0,
\end{equation*}

\begin{equation*}\label{s43}
\tau z_{t}(\rho, t)+z_{\rho}(\rho,t)=0,\quad (\rho,t)\in(0,1)\times(0,+\infty),
\end{equation*}

\begin{equation*}\label{s44}
u(0)=u_{0}, u'(0)=u_{1}, \theta(0)=\theta_{0},
\end{equation*}

\begin{equation*}\label{s45}
z(\rho,0)=\psi(-\tau\rho),\quad \rho\in(0,1),
\end{equation*}

\begin{equation*}\label{s46}
\ds z(0,t)=A^{\frac{\alpha}{2}}\theta(t),\quad t>0.
\end{equation*}

\noindent Define
$$U=(u,u',\theta,z)^{\top},$$
then, problem (\ref{s1}) can be formulated as a first order system of the form

\begin{equation}\label{s47}
\left\{
\begin{array}{ll}
U'=\mathcal{A}_{\alpha,\beta}U, & \\
\ds U(0)=\big(x_{0},x_{1},\theta_{0},\psi(-\tau \cdot)\big)^{\top},
\end{array}
\right.
\end{equation}
where the operator $\mathcal{A}_{\alpha,\beta}$ is defined by
$$\mathcal{A}_{\alpha,\beta}
\left(
\begin{array}{c}
u \\
v \\
\theta \\
z \\
\end{array}
\right)= \left(
\begin{array}{c}
v \\
-A^{1/2}\left(A^{1/2} u-A^{\beta-\frac{1}{2}}\theta\right)  \\
\ds-\kappa A^{\frac{\alpha}{2}}\big(z(1)+\frac{a}
{\kappa}A^{\frac{\alpha}{2}}\theta+\frac{1}{\kappa} A^{\beta-\frac{\alpha}{2}}v\big) \\
\ds-\frac{1}{\tau}z_{\rho} \\
  \end{array}
\right),$$
with domain
$$D(\mathcal{A}_{\alpha,\beta})=\left\{
\begin{array}{c}
(u,v,\theta,z)^{\top}\in D(A^{\frac{1}{2}})\times D(A^{\frac{1}{2}})
\times D(A^{\frac{\alpha}{2}})\times H^{1}\big((0,1),H\big): \quad z(0)=A^{\frac{\alpha}{2}}\theta, \\
\ds  A^{1/2} u-A^{\beta-\frac{1}{2}}\theta\in D(A^{1/2}), \quad\textrm{ and }\quad z(1)+\frac{a}
{\kappa}A^{\frac{\alpha}{2}}\theta+\frac{1}{\kappa} A^{\beta-\frac{\alpha}{2}}v\in D(A^{\frac{\alpha}{2}})
\end{array}
\right\},$$
in the Hilbert space
$$\mathcal{H}=D(A^{\frac{1}{2}})\times H\times H\times L^{2}\big((0,1),H\big),$$
equipped with the scalar product
$$\ds \big((u,v,\theta,z)^\top,(u_1,v_1,\theta_1,z_1)^\top\big)_{\mathcal{H}}
=\big(A^{\frac{1}{2}}u,A^{\frac{1}{2}}u_1\big)_{H}
+(v,v_1)_{H}+(\theta,\theta_1)_{H}+\xi\int_{0}^{1}(z,z_1)_{H}d\rho,$$
where $\xi$ is a parameter that will be fixed later on. 

\begin{lemma}\rm\label{lemma22"}
Suppose that $(\beta,\alpha)\in Q$ and assume that $a\geq \kappa$. Let
$$\ds\mathcal{J}_{a,\kappa,\tau}=\big[\tau\big(a-\sqrt{a^2-\kappa^2}\big),\tau\big(a+\sqrt{a^2-\kappa^2}\big)\big].$$
Then, for
$\ds\xi\in\mathcal{J}_{a,\kappa,\tau}$, the operator
$\mathcal{A}_{\alpha,\beta}$ is dissipative in $\mathcal{H}$.
\end{lemma}

\begin{proof}
Take $U=(u,v,\theta,z)^{\top}\in D(\mathcal{A}_{\alpha,\beta})$.
Then, we have
\begin{eqnarray*}
\ds\big(\mathcal{A}_{\alpha,\beta}U,U\big)_{\mathcal{H}}&=&
\big(A^{\frac{1}{2}}v,A^{\frac{1}{2}}u\big)_{H}-
\big(A^{\frac{1}{2}}u,A^{\frac{1}{2}}v\big)_{H}
+(A^{\beta-\frac{1}{2}}\theta,A^{\frac{1}{2}}v)_{H}-(A^{\beta-\frac{\alpha}{2}}v+\kappa z(1),A^{\frac{\alpha}{2}}\theta)_{H} \\
&-&a\|A^{\frac{\alpha}{2}}\theta\|_{H}^{2}-
\frac{\xi}{\tau}\int_{0}^{1}(z_{\rho},z)_{H}d\rho.
\end{eqnarray*}

\noindent Using the Young's inequality and that
$\ds z(0)=A^{\frac{\alpha}{2}}\theta$,
we find that, for every $\varepsilon>0$,

\begin{eqnarray*}
Re\big(\mathcal{A}_{\alpha,\beta}U,U\big)_{\mathcal{H}}&\leq&
\bigg(\frac{\kappa}{2\varepsilon}-\frac{\xi}{2\tau}\bigg)\|z(1)\|_H^{2}+
\bigg(\frac{\varepsilon\kappa}{2}-a+\frac{\xi}{2\tau}\bigg)
\|A^{\frac{\alpha}{2}}\theta\|_H^{2}.
\end{eqnarray*}

\noindent An easy computation shows that

$$\ds\frac{\kappa\tau}{\xi}\leq\frac{2}{\kappa}\left(a-\frac{\xi}{2\tau}\right)
\Longleftrightarrow\xi\in\mathcal{J}_{a,\kappa,\tau}.$$
Hence, once we have $\xi\in\mathcal{J}_{a,\kappa,\tau}$,
one can choose
$\ds \varepsilon=:\varepsilon_0\in\left[\frac{\kappa\tau}{\xi},\frac{2}{\kappa}
\big(a-\frac{\xi}{2\tau}\big)\right]$.
Thus,
\begin{equation}\label{s48"}
Re\big(\mathcal{A}_{\alpha,\beta}U,U\big)_{\mathcal{H}}\leq 0,\quad
\forall \, U \in D(\mathcal{A}_{\alpha,\beta}),
\end{equation}
that is, $\mathcal{A}_{\alpha,\beta}$ is dissipative
\end{proof}

In fact, we will prove that $\mathcal{A}_{\alpha,\beta}$ generates a $\mathcal{C}_0$-semigroup of contractions. For this we use the following result together with Lemma \ref{lemma22"},

\begin{lemma}\rm\label{lemma23"}
For $(\beta,\alpha)\in Q$ and
$\ds\xi>0$ we have
\begin{equation}\label{s49"}
0\in\rho\big(\mathcal{A}_{\alpha,\beta}\big).
\end{equation}
\end{lemma}

\begin{proof}
\noindent Fix $F=(f,g,p,h)^{\top}\in \mathcal{H}$. We seek
a unique solution $U=(u,v,\theta,z)^{\top}\in D(\mathcal{A}_{\alpha,\beta})$ of
\begin{equation*}
\mathcal{A}_{\alpha,\beta}U=F,\label{3.9'}
\end{equation*}

that is, verifying
\begin{equation}\label{s50"}
\left\{
\begin{array}{ll}
v=f, &  \\
\ds-A^{1/2}\big(A^{1/2} u-A^{\beta-\frac{1}{2}}\theta\big)=g, &  \\
-\kappa A^{\frac{\alpha}{2}}\big(z(1)+\frac{a}
{\kappa}A^{\frac{\alpha}{2}}\theta+\frac{1}{\kappa} A^{\beta-\frac{\alpha}{2}}v\big)=p, &  \\
\ds-\frac{1}{\tau}z_\rho=h. &
\end{array}
\right.
\end{equation}

First, to determine z, we set $\ds z(0)=A^{\frac{\alpha}{2}}\theta$. Then,
by $(\ref{s50"})_{4}$, we get
\begin{equation*}\label{s52"}
\ds z(\rho)=A^{\frac{\alpha}{2}}\theta-\tau
\int_{0}^{\rho}h(s)ds.
\end{equation*}
and
\begin{equation*}\label{s53"}
z(1)=A^{\frac{\alpha}{2}}\theta-z_{0}
\end{equation*}
with
\begin{equation*}\label{s54"}
z_{0}=\tau \int_{0}^{1}h(s)
ds\in L^{2}\big((0,1),H\big).
\end{equation*}

Second, define
\begin{eqnarray*}
v:=f,\;\;\theta:=-\frac{1}{\kappa+a}A^{-\alpha/2}\left( A^{-\alpha/2}p+A^{\beta-\alpha/2}f-\kappa z_0\right) ,\\
u:=-A^{-1/2}\left( A^{-1/2}g+\frac{1}{\kappa+a}A^{\beta-1/2-\alpha/2}\Big(A^{-\alpha/2}p+A^{\beta-\alpha/2}f-\kappa z_0\Big
)\right).
\end{eqnarray*}

We have $U=(u,v,\theta,z)^{\top}\in D(\mathcal{A}_{\alpha,\beta})$, $\mathcal{A}_{\alpha,\beta}U=F$ and there exists a constant $c>0$ such that $\|U\|_{\mathcal{H}}\leq c\|F\|_{\mathcal{H}}$ , where we have used that $2\beta\leq 1+\alpha$.
\end{proof}

Now, combining the results in Lemma \ref{lemma22"} and Lemma \ref{lemma23"}, to derive the following result:

\begin{proposition}\label{p4.1"}
Suppose that $(\beta,\alpha)\in Q$, assume that $a\geq\kappa$ and let
$\ds\xi\in\mathcal{J}_{a,\kappa,\tau}$.
 Then, the system (\ref{s1}) is well posed.
More precisely, the operator $\mathcal{A}_{\alpha,\beta}$ generates a
$\mathcal{C}_0$-semigroup of contractions on $\mathcal{H}$.
Thus, for every
$(u_0,v_0,\theta_0,z_0)\in \mathcal{H}$, there exists a unique solution
$(u,v,\theta,z)\in\mathcal{C}\big((0,+\infty),\mathcal{H}\big)$ of (\ref{s47}).
In addition, if the initial datum
$(u_0,v_0,\theta_0,z_0)\in D(\mathcal{A}_{\alpha,\beta})$,
then the solution satisfies
$(u,u^\prime,\theta,z)\in\mathcal{C}\big((0,+\infty),D(\mathcal{A}_{\alpha,\beta})\big)\cap
\mathcal{C}^{1}\big((0,+\infty),\mathcal{H}\big)$.
\end{proposition}



\subsection{Stability of the delayed coupled system (\ref{s1})}
In this section we will prove that the semigroup $e^{t\mathcal{A}_{\alpha,\beta}}$ is exponentially stable in region $S$ and polynomially stable in region $S_1\cup S_2$. The proofs  are based on the frequency domain results for uniform and polynomial stability of a semigroup of contractions.

The first characterizes a $\mathcal{C}_0$-semigroup of contractions to be exponentially stable, it is an invariant of that proposed in Lemma \ref{lemma exp.stab}, and due to \cite{Gea78}.  
\begin{lemma} \label{lem3.3}
A $\mathcal{C}_0$-semigroup of contractions $e^{t\mathcal{L}}$ on a Hilbert space $G$ satisfies
$$\|e^{t\mathcal{L}}\|_{\mathcal{L}(G)}\leq Ce^{-wt}$$
for some constants $C>0$ and $w>0$ if and only if

\begin{equation}\label{s24''}
i\mathbb{R}\subseteq \rho (\mathcal{L})
\end{equation}

and

\begin{equation}\label{s25''}
\underset{\lambda\in\mathbb{R},\,\left| \lambda \right| \rightarrow \infty }{\limsup }\left\| (i\lambda
-\mathcal{L})^{-1}\right\|<\infty .
\end{equation}
\end{lemma}

The second, due to Borichev and Tomilov \cite{BoTo10} (see also \cite{BattyDuck}), concerns the polynomial stability of a $\mathcal{C}_0$ semigroup of contractions.
 
\begin{lemma}\label{lem3.4}
A $\mathcal{C}_0$ semigroup of contractions $e^{t\mathcal{L}}$ on a Hilbert space $G$ satisfies 
\begin{equation*}
\left\| e^{\mathcal{L}t}\right\|_{\mathcal{L}(\mathcal{D}(\mathcal{L}),\mathcal{H})}
\leq \frac{C}{t^{\frac{1}{\alpha }}},\;\;\;t>0
\end{equation*}
for some constant $C>0$ and for $\alpha >0$ if, and only if, (\ref{s24''})
holds and 
\begin{equation}\label{s25'''}
\underset{\lambda\in\mathbb{R},\,\left| \lambda \right| \rightarrow \infty }{\limsup } \frac{1}{%
\left| \lambda \right| ^{\alpha }}\left\| (\mathbf{i}\lambda -\mathcal{L})^{-1}\right\| <\infty.
\end{equation} 
In this case, we say that the semigroup is of order $1/\alpha$.
\end{lemma}

\subsubsection{Exponential stability} 
\begin{theorem}\rm\label{theorem41}
Assume that $a>\kappa$. Then for $\xi\in \mathring{\mathcal{J}}_{a,\kappa,\tau}$,  
 the semigroup $e^{t\mathcal{A}_{\alpha,\beta}}$ is exponentially stable in region $S$.
\end{theorem}
\begin{proof}
By Lemma \ref{lem3.3} we need to check conditions (\ref{s24''}) and (\ref{s25''}).

\medskip

Suppose that (\ref{s25''}) is not true, then there exists a sequence $(\lambda_n)$ of real numbers, with $\lambda_n\rightarrow \infty$ (without loss of generality, we suppose that $\lambda_n>0$) and a sequence
of unit vectors $U_n=(u_n,v_n,\theta_n,z_n)^\top\in D(\mathcal{A}_{\alpha,\beta})$
with
\begin{equation}\label{s60"}
\underset{n\rightarrow+\infty}\lim\|\big(\mathbf{i}\lambda_n I-\mathcal{A}_{\beta,\alpha}\big)U_n\|_{\mathcal{H}}=0,
\end{equation}

\noindent that is

\begin{equation}\label{s61"}
\|A^{\frac{1}{2}}u_n\|_H^{2}+\|v_n\|_H^{2}+\|\theta_n\|_H^{2}+\xi\int_{0}^{1}\|z_n(\rho)\|_Hd\rho=1,
\quad\forall n\in\N
\end{equation}

and

\begin{equation}\label{s62"}
\ds\mathbf{i}\lambda_n A^{\frac{1}{2}}u_n-A^{\frac{1}{2}}v_n\rightarrow0 \textrm{ in } H,
\end{equation}

\begin{equation}\label{s63"}
\ds\mathbf{i}\lambda_n v_n+A^{1/2}\big(A^{1/2}u_n-A^{\beta-\frac{1}{2}}\theta_n\big)\rightarrow0 \textrm{ in } H,
\end{equation}

\begin{equation}\label{s64"}
\ds\mathbf{i}\lambda_n\theta_n+\kappa A^{\frac{\alpha}{2}}
\big(z_n(1)+\frac{a}{\kappa}A^{\frac{\alpha}{2}}\theta_n
+\frac{1}{\kappa} A^{\beta-\frac{\alpha}{2}}v_n\big)\rightarrow0 \textrm{ in } H,
\end{equation}

\begin{equation}\label{s65"}
\ds\mathbf{i}\lambda_nz_n+\frac{1}{\tau}\partial_{\rho}z_n\equiv
h_n\rightarrow0 \textrm{ in }L^2\big((0,1)H\big).
\end{equation}
Since $\xi\in \mathring{\mathcal{J}}_{a,\kappa,\tau}$,
one can choose, as in the proof of Lemma \ref{lemma22"},
$\ds \varepsilon=:\varepsilon_0 \in \left]\frac{\kappa\tau}{\xi},
\frac{2}{\kappa}\left(a-\frac{\xi}{2\tau}\right)\right[$. 

\medskip

Therefore, we get

\begin{eqnarray}
\textrm{Re}\bigg((\mathbf{i}\lambda_n I-\mathcal{A}_{\alpha,\beta})U_n,U_n\bigg)_{\mathcal{H}}
&=&-\textrm{Re}
(\mathcal{A}_{\alpha,\beta}U_n,U_n)_{\mathcal{H}} \notag \\
 &\geq&\left( \frac{\xi}{2\tau}-
\frac{\kappa}{2\varepsilon_0}\right) \|z_n(1)\|_H^{2}
+\left(a-\frac{\xi}{2\tau}-\frac{\varepsilon_0\kappa}{2}\right)
\|A^{\frac{\alpha}{2}}\theta_n\|_H^{2}\geq 0.\label{s66-"}
\end{eqnarray}
 It follows (taking in mind (\ref{s60"})) that

\begin{equation}\label{s66"}
 z_n(1)\underset{n\rightarrow+\infty}\rightarrow0\quad \text{ and }\quad
A^{\frac{\alpha}{2}}\theta_n\underset{n\rightarrow+\infty}\rightarrow0\textrm{ in }H.
\end{equation}
Hence,

\begin{equation}\label{s67"}
\theta_n\underset{n\rightarrow+\infty}\rightarrow0\textrm{ in }H.
\end{equation}


Now, taking the inner product of (\ref{s62"}) with $\ds \frac{1}{\lambda_n} A^{\frac{1}{2}}u_n$ and (\ref{s63"}) with $\ds \frac{1}{\lambda_n} v_n$  in $H$, respectively,  in $H$ to get 
\begin{equation}\label{s69"}
\mathbf{i}\|A^{\frac{1}{2}}u_n\|_H^{2}-\frac{1}{\lambda_n}(A^{\frac{1}{2}}v_n,A^{\frac{1}{2}}u_n)_{H}
\underset{n\rightarrow+\infty}\rightarrow0.
\end{equation}


and

\begin{equation}\label{s71"}
-\mathbf{i}\|v_n\|_H^{2}+\frac{1}{\lambda_n}(A^{\frac{1}{2}}v_n,A^{\frac{1}{2}}u_n)_{H}-\frac{1}{\lambda_n}(A^{\frac{1}{2}}v_n,A^{\beta-\frac{1}{2}}\theta_n)_{H}
\underset{n\rightarrow+\infty}\rightarrow0\textrm{ in }H.
\end{equation}

Note that $\frac{1}{\lambda_n}A^{1/2}v_n$ is bounded (by (\ref{s62"}) and the boundedness of $A^{1/2}u_n$) and $A^{\frac{\alpha}{2}}\theta_n\underset{n\rightarrow+\infty}\rightarrow0$ (by (\ref{s66"}) and the fact that $\beta-\frac{1}{2}\leq \frac{\alpha}{2}$). Thus
\begin{equation}\label{3.19}
\frac{1}{\lambda_n}(A^{\frac{1}{2}}v_n,A^{\beta-\frac{1}{2}}\theta_n)_{H}
\underset{n\rightarrow+\infty}\rightarrow 0.
\end{equation}



 Summing up (\ref{s69"}), (\ref{s71"}) taking in mind (\ref{3.19}),  we get,

\begin{equation}\label{s83"}
\ds\|A^{\frac{1}{2}}u_n\|_H^{2}- \|v_n\|_H^{2}
\underset{n\rightarrow+\infty}\rightarrow 0.
\end{equation}
\begin{lemma}
\begin{equation}\label{s83"""}
\|v_n\|_H^{2}
\underset{n\rightarrow+\infty}\rightarrow0\textrm{ in }H,
\end{equation}
\end{lemma}
\begin{proof}

Recall here that $(\beta,\alpha)\in S$.

Taking the inner product of (\ref{s63"}) with $A^{-\beta}\theta_n$ to get
\begin{equation}\label{s83'"}
\mathbf{i}\lambda_n(v_n,A^{-\beta}\theta_n)_H+(A^{1/2}u_n,A^{-\beta+\frac{1}{2}}\theta_n)_H-(A^{\beta-\frac{1}{2}}\theta_n,A^{-\beta+\frac{1}{2}}\theta_n)_{H}
\underset{n\rightarrow+\infty}\rightarrow0\textrm{ in }H.
\end{equation}
Using that $A^{\frac{\alpha}{2}}\theta_n=o(1)$, $\beta-\frac{1}{2}\leq\frac{\alpha}{2}$ and $A^{\frac{1}{2}}u_n$ is bounded, we deduce that the second and the third terms in (\ref{s83'"}) converge to zero, then
\begin{equation}\label{s83''"}
\mathbf{i}\lambda_n(v_n,A^{-\beta}\theta_n)_H
\underset{n\rightarrow+\infty}\rightarrow0\textrm{ in }H.
\end{equation}
Taking the inner product of (\ref{s64"}) with $A^{-\beta}v_n$ to get
\begin{equation}\label{s83'''"}
\mathbf{i}\lambda_n(\theta_n,A^{-\beta}v_n)_H+\kappa(z_n(1),A^{-\beta+\frac{\alpha}{2}}v_n)_H+a(A^{\frac{\alpha}{2}}\theta_n,A^{-\beta+\frac{\alpha}{2}}v_n)_H+\|v_n\|_H^{2}
\underset{n\rightarrow+\infty}\rightarrow0\textrm{ in }H.
\end{equation}
Since $A^{\frac{\alpha}{2}}\theta_n=o(1)$, $z_n(1)=o(1)$, $A^{-\beta+\frac{\alpha}{2}}v_n=O(1)$ (because $-\beta+\frac{\alpha}{2}\leq 0$ and $v_n$ is bounded), then taking into account (\ref{s83''"}), we deduce that  
\begin{equation}\label{s83''""}
\|v_n\|_H^{2}
\underset{n\rightarrow+\infty}\rightarrow0\textrm{ in }H.
\end{equation}
\end{proof}
By (\ref{s83"}) we have also
$$\|A^{1/2}u_n\|_H^{2}
\underset{n\rightarrow+\infty}\rightarrow0\textrm{ in }H,$$
By integration of the identity (\ref{s65"}) and using (\ref{s66"}), we get
\begin{equation*}
\ds\|z_n\|_{L^{2}\big((0,1),H\big)}^{2}\rightarrow0.
\end{equation*}

\noindent All together, we have shown that $\{U_{n}\}$ converges to $0$
in $\mathcal{H}$. This clearly contradicts (\ref{s61"}).
Thus, condition (\ref{s25''}) holds.

Now, to show (\ref{s24''}), we again use a contradiction argument. Suppose that (\ref{s24''}) is not true. Since $0\in\rho(\mathcal{A}_{\alpha,\beta})$, it can be proved that, there exists $\lambda \in \mathbb{R}$ with $\lambda \neq 0$, a sequence $(\lambda_n)$ of real numbers, with $\lambda_n\rightarrow \lambda$  
and a sequence
of unit vectors $U_n=(u_n,v_n,\theta_n,z_n)^\top\in D(\mathcal{A}_{\alpha,\beta})$
such that
\begin{equation}\label{s60""}
\underset{n\rightarrow+\infty}\lim\|\big(\mathbf{i}\lambda_n I-\mathcal{A}_{\alpha,\beta}\big)U_n\|_H=0.
\end{equation}
A repetition of the above argument leads to the same contradiction since above we did not use the condition $\lambda_n\rightarrow \infty$. So the proof of Theorem \ref{theorem41} is completed. 
\end{proof}

 By specifying the regions $S_1$ and  $S_2$ as follow,
$$S_1:=\{ (\beta,\alpha)\in[0,1]\times[\frac{1}{2},1]\mid  0<\alpha-2\beta\}$$ and
$$ S_2:=\{ (\beta,\alpha)\in[0,1]\times[0,\frac{1}{2}]\mid   \alpha +2\beta<1\},$$
we have the following result.
\begin{theorem}\rm\label{theorem41'}
Assume that $a>\kappa$. Then for $\xi\in \mathring{\mathcal{J}}_{a,\kappa,\tau}$,  
 the semigroup $e^{t\mathcal{A}_{\alpha,\beta}}$ has the following stability properties:

(i) In $S_1 $, it is polynomially stable of order $\frac{1}{2(\alpha-2\beta)}$;

(ii) In $S_2$, it is polynomially stable of order $\frac{1}{2-2(\alpha+2\beta)}$.
\end{theorem}
\begin{proof}
We prove the two cases (i) and (ii) simultaneously. So in the sequel, $\gamma=2(\alpha-2\beta)$ in the case (i) and $\gamma=2-(\beta+2\alpha)$ in the case (ii). By Lemma \ref{lem3.4} we need to check conditions (\ref{s24''}) and (\ref{s25'''}).

Suppose that (\ref{s25'''}) is not true, then there exists a sequence $(\lambda_n)$ of real numbers, with $|\lambda_n|\rightarrow \infty$  and a sequence
of unit vectors $U_n=(u_n,v_n,\theta_n,z_n)^\top\in D(\mathcal{A}_{\alpha,\beta})$
with
\begin{equation}\label{s60"'}
\underset{n\rightarrow+\infty}\lim|\lambda_n|^{\gamma}\|\big(\mathbf{i}\lambda_n I-\mathcal{A}_{\beta,\alpha}\big)U_n\|_H=0,
\end{equation}

\noindent that is

\begin{equation}\label{s61"'}
\|A^{\frac{1}{2}}u_n\|_H^{2}+\|v_n\|_H^{2}+\|\theta_n\|_H^{2}+\xi\int_{0}^{1}\|z_n(\rho)\|_Hd\rho=1,
\quad\forall n\in\N
\end{equation}

and

\begin{equation}\label{s62"'}
\ds|\lambda_n|^{\gamma}\left( \mathbf{i}\lambda_n A^{\frac{1}{2}}u_n-A^{\frac{1}{2}}v_n\right) \rightarrow0 \textrm{ in } H,
\end{equation}

\begin{equation}\label{s63"'}
\ds|\lambda_n|^{\gamma}\left( \mathbf{i}\lambda_n v_n+A^{1/2}\big(A^{1/2}u_n-A^{\beta-\frac{1}{2}}\theta_n\big)\right) \rightarrow0 \textrm{ in } H,
\end{equation}

\begin{equation}\label{s64"'}
\ds|\lambda_n|^{\gamma}\left( \mathbf{i}\lambda_n\theta_n+\kappa A^{\frac{\alpha}{2}}
\big(z_n(1)+\frac{a}{\kappa}A^{\frac{\alpha}{2}}\theta_n
+\frac{1}{\kappa} A^{\beta-\frac{\alpha}{2}}v_n\big)\right) \rightarrow0 \textrm{ in } H,
\end{equation}

\begin{equation}\label{s65"'}
\ds|\lambda_n|^{\gamma}\left( \mathbf{i}\lambda_nz_n+\frac{1}{\tau}\partial_{\rho}z_n\right) \equiv
h_n\rightarrow0 \textrm{ in }L^2\big((0,1)H\big).
\end{equation}
Since $\xi\in \mathring{\mathcal{J}}_{a,\kappa,\tau}$,
one can choose, as in the proof of Lemma \ref{lemma22"},
$\ds \varepsilon=:\varepsilon_0 \in \left]\frac{\kappa\tau}{\xi},
\frac{2}{\kappa}\left(a-\frac{\xi}{2\tau}\right)\right[$. 

\medskip

Therefore, we get

$\textrm{Re}\bigg(|\lambda_n|^{\gamma}(\mathbf{i}\lambda_n I-\mathcal{A}_{\alpha,\beta})U_n,U_n\bigg)_{\mathcal{H}}
=-\textrm{Re}
|\lambda_n|^{\gamma}(\mathcal{A}_{\alpha,\beta}U_n,U_n)_{\mathcal{H}}$
\begin{eqnarray}
\geq \big(\frac{\xi}{2\tau}-
\frac{\kappa}{2\varepsilon_0}\big)|\lambda_n|^{\gamma}\|z_n(1)\|_H^{2}
+\left(a-\frac{\xi}{2\tau}-\frac{\varepsilon_0\kappa}{2}\right)
|\lambda_n|^{\gamma}\|A^{\frac{\alpha}{2}}\theta_n\|_H^{2}\geq 0.\label{s66-"'}
\end{eqnarray}
 It follows (by (\ref{s60"'})) that

\begin{equation}\label{s66"'}
|\lambda_n|^{\frac{\gamma}{2}}\| z_n(1)\|\underset{n\rightarrow+\infty}\rightarrow0\quad \text{ and }\quad
|\lambda_n|^{\frac{\gamma}{2}}\|A^{\frac{\alpha}{2}}\theta_n\|\underset{n\rightarrow+\infty}\rightarrow0\textrm{ in }H
\end{equation}
and in particular, 
\begin{equation}\label{s67"'}
|\lambda_n|^{
\frac{\gamma}{2}}\|\theta_n\|\underset{n\rightarrow+\infty}\rightarrow0\textrm{ in }H.
\end{equation}

First, we show that $\|v_n\|\underset{n\rightarrow+\infty}\rightarrow0$. As in \cite{HaLi13}, acting the bounded operator $|\lambda_n|^{-\frac{\gamma}{2}}A^{-\frac{\alpha}{2}}$ on (\ref{s64"'}) and applying (\ref{s66"'}), to obtain
\begin{equation}\label{s64bis"'}
\ds \mathbf{i}\lambda_n|\lambda_n|^{\frac{\gamma}{2}}A^{-\frac{\alpha}{2}}\theta_n
+|\lambda_n|^{\frac{\gamma}{2}} A^{\beta-\frac{\alpha}{2}}v_n \rightarrow0 \textrm{ in } H,
\end{equation}
which is exactly equation (2.7) and equation (2.30) in \cite{HaLi13}. Thus, the rest of the estimate of $\|v_n\|$ will be the same as in \cite{HaLi13}. So we omit the details. 



Second, we estimate $\|A^{\frac{1}{2}}u_n\|_H$. As in the proof of Theorem \ref{theorem41}, we take the inner product of (\ref{s62"'}) with $\frac{1}{\lambda_n|\lambda_n|^{\gamma}}A^{1/2}u_n$ in $H$ and take the inner product of  $\frac{1}{\lambda_n|\lambda_n|^{\gamma}}v_n$ with (\ref{s63"'}) in $H$, respectively, then summing yields

\begin{equation}\label{s83"'}
\ds\mathbf{i}\|A^{\frac{1}{2}}u_n\|_H^{2}-\mathbf{i} \|v_n\|_H^{2}-\frac{1}{\lambda_n}(A^{\frac{1}{2}}v_n,A^{\beta-\frac{1}{2}}\theta_n)_{H}
\underset{n\rightarrow+\infty}\rightarrow0.
\end{equation}

Using the boundedness of $\frac{1}{\lambda_n}A^{1/2}v_n$ (given by (\ref{s61"'}) and (\ref{s62"'})) together with (\ref{s66"'}) and that $\beta-\frac{1}{2}\leq \frac{\alpha}{2}$, we deduce
that the last term in the left hand side of (\ref{s83"'}) goes to zero as $n$ goes to infinity. Hence (\ref{s83"'}) is reduced to
\begin{equation}\label{s83bis"}
\ds\|A^{\frac{1}{2}}u_n\|_H^{2}
\underset{n\rightarrow+\infty}\rightarrow 0
\end{equation}
where we have used the above estimate of $\|v_n\|_H$.

By integration of the identity (\ref{s65"'}) and using (\ref{s66"'}), we get
\begin{equation*}
\ds\|z_n\|_{L^{2}\big((0,1),H\big)}^{2}\rightarrow0.
\end{equation*}

\noindent All together, we have shown that $\{U_{n}\}$ converges to $0$
in $\mathcal{H}$. This clearly contradicts (\ref{s61"'}).
Thus, condition (\ref{s25'''}) holds.

\medskip

To show (\ref{s24''}), we again use a contradiction argument. Suppose that (\ref{s24''}) is not true. Since $0\in\rho(\mathcal{A}_{\alpha,\beta})$, it can be proved that, there exists $\lambda \in \mathbb{R}$ with $\lambda \neq 0$, a sequence $(\lambda_n)$ of real numbers, with $\lambda_n\rightarrow \lambda$  
and a sequence
of unit vectors $U_n=(u_n,v_n,\theta_n,z_n)^\top\in D(\mathcal{A}_{\alpha,\beta})$
such that $\underset{n\rightarrow+\infty}\lim\|\big(\mathbf{i}\lambda_n I-\mathcal{A}_{\alpha,\beta}\big)U_n\|_H=0$, which implies in particular,
\begin{equation}\label{s60""'}
\underset{n\rightarrow+\infty}\lim|\lambda_n|^\gamma\|\big(\mathbf{i}\lambda_n I-\mathcal{A}_{\alpha,\beta}\big)U_n\|_H=0.
\end{equation}
A repetition of the above argument leads to the same contradiction since above we did not use the condition $|\lambda_n|\rightarrow \infty$. So the proof of Theorem \ref{theorem41'} is completed. 
\end{proof}
\begin{remark}
The damping term $aA^\alpha \theta(t)$ added to the heat equation has the same order of $\kappa A^\alpha \theta(t-\tau)$. So we think  that we cannot expect to have better stability results of system (\ref{s1}) than those obtained in the case $\tau=0$. In particular we conjecture that system (\ref{s1}) is only polynomially stable in the region $S_1\cup S_2$.
\end{remark}
\subsection{Some related systems} Suppose that $A^{\alpha}=BB^*$
where $B:D(B):H\rightarrow H$ is a closed densely defined linear unbounded operator.The 
 assumption 
$$B^*\theta(t-\tau)=g_0(t),  \quad t\in(0,\tau)$$
is given  instead of the assumption
$$A^{\alpha/2}\theta(t-\tau)=g_0(t),  \quad t\in(0,\tau).$$
 
Precisely, we consider the following system
 \begin{equation}\label{s1''}
\left\{
\begin{array}{ll}
u''(t)+ Au(t)-C\theta(t)=0, & \quad t\in(0,+\infty), \\
\theta'(t)+\kappa A^\alpha \theta(t-\tau)+aA^\alpha \theta(t)+C^{*}u'(t)=0, & \quad t\in(0,+\infty), \\
u(0)=u_{0}, u'(0)=u_{1}, \theta(0)=\theta_{0}, &  \\
\ds B^{*}\theta(t-\tau)=\psi(t-\tau), & \quad t\in(0,\tau).
\end{array}
\right.
\end{equation}
where $C:D(C) \subset H\rightarrow H$ is a closed densely defined linear operator such that $CC^*=A^{2\beta}$.
 
We suppose that (\ref{s1''}) satisfies property $(\mathcal{H}_1)$ where
 
 $(\mathcal{H}_1)$:\;\;(i)\;\;There  exists $c_1 > 0$ such that
 \begin{equation}\label{coer2}
     \|B^{*}v\|_H\geq c_1\|v\|_{H},\;\;\forall\
v\in D(B^{*}).
 \end{equation}
 
 (ii)\;\;  $D(A^{1/2})\subset D(C^*)$ and $D(B^{*})\subset D(A^{-1/2}C)$ and there exists $c_2 > 0$ such that $$\|A^{-1/2}Cv\|_H\leq c_2\|B^*v\|_{H},\;\;\forall\
v\in D(B^{*}).$$  
\subsubsection{Well-posedness}
Here we take
$$\ds z(\rho,t):=B^{*}\theta(t-\tau\rho),\quad\rho\in(0,1), t>0.$$
Then, problem (\ref{s1''}) is rewritten as

\begin{equation}\label{s0'''}
\left\{
\begin{array}{ll}
u''(t)+Au(t)-C\theta(t)=0, &  t>0, \\
\theta'(t)+\kappa Bz(1,t)+aBB^*\theta(t)+C^{*}u'(t)=0, & t>0, \\
\tau z_{t}(t,\rho)+z_{\rho}(\rho,t)=0, & (\rho,t)\in(0,1)\times(0,+\infty), \\
u(0)=u_{0}, u'(0)=u_{1}, \theta(0)=\theta_{0}, & \\
z(\rho,0)=\psi(-\tau\rho),& \rho\in(0,1),\\
z(0,t)=B^*\theta(t),&\quad t>0.
\end{array}
\right.
\end{equation}

Define
$U=(u,u',\theta,z)^{\top}$,
then problem (\ref{s0'''}) can be formulated as a first order system of the form

\begin{equation}\label{s7'}
\left\{
\begin{array}{ll}
U'=\mathcal{A}_{\alpha,\beta}U, & \\
\ds U(0)=\big(u_{0},u_{1},\theta_{0},\psi(-\tau \cdot)\big)^{\top},
\end{array}
\right.
\end{equation}
where the operator $\mathcal{A}_{\alpha,\beta}$ is defined by
$$\mathcal{A}_{\alpha,\beta}
\left(
\begin{array}{c}
u \\
v \\
\theta \\
z \\
\end{array}
\right)=
\left(
\begin{array}{c}
v \\
\ds-A^{1/2}\big(A^{1/2}u-A^{-1/2}C\theta\big) \\
-B\big(\kappa z(1)+aB^*\theta\big)-C^{*}v \\
\ds-\frac{1}{\tau}z_{\rho} \\
  \end{array}
\right),$$
with domain
$$D(\mathcal{A}_{\alpha,\beta})=\left\{
\begin{array}{c}
(u,v,\theta,z)^{\top}\in D(A^{1/2})\times D(A^{1/2})
\times D(B^{*})\times H^{1}\big((0,1),H\big): \\
\ds z(0)=B^*\theta,\quad A^{1/2}u-A^{-1/2}C\theta\in D(A^{1/2}),\quad\textrm{ and }\quad aB^*\theta+\kappa z(1)\in D(B)
\end{array}
\right\},$$
in the Hilbert space
$$\mathcal{H}=D(A^{1/2})\times H\times H\times L^{2}\big((0,1),H\big),$$
equipped with the scalar product
$$\ds \big((u,v,\theta,z)^\top,(u_1,v_1,\theta_1,z_1)^\top\big)_{\mathcal{H}}
=\big(A^{1/2}u,A^{1/2}u_1\big)_{H}
+(v,v_1)_{H}+(\theta,\theta_1)_{H}+\xi\int_{0}^{1}(z,z_1)_{H}d\rho,$$
where $\xi>0$ is a positive constant.

\begin{lemma}\rm\label{lemma22}
Assume that $a\geq \kappa$. 
Then, for
$\ds\xi\in\mathcal{J}_{a,\kappa,\tau}$, the operator
$\mathcal{A}_{\alpha,\beta}$ is dissipative in $\mathcal{H}$.
\end{lemma}

\begin{proof}
Take $U=(u,v,\theta,z)^{\top}\in D(\mathcal{A}_{\alpha,\beta})$.
Then, we have
\begin{eqnarray*}
\ds\big(\mathcal{A}_{\alpha,\beta}U,U\big)_{\mathcal{H}}&=&
\big(A^{\frac{1}{2}}v,A^{\frac{1}{2}}u\big)_{H}-
\big(A^{\frac{1}{2}}u,A^{\frac{1}{2}}v\big)_{H}
+ (A^{-1/2}C\theta,A^{1/2}v)_{H}-(\kappa z(1),B^*\theta)_{H}-(C^{*}v,\theta)_{H} \\
&-&a\|B^{*}\theta\|_{H}^{2}-
\frac{\xi}{\tau}\int_{0}^{1}(z_{\rho},z)_{H}d\rho.
\end{eqnarray*}
Since $C$ is densely defined, we have $\big(A^{-1/2}C\big)^*=C^*A^{-1/2}$, then $(A^{-1/2}C\theta,A^{1/2}v)_{H}=(\theta,C^{*}v)_{H}$.

\medskip

Then, with using the Young's inequality and that
$\ds z(0)=A^{\frac{\alpha}{2}}\theta$,
we find that, for every $\varepsilon>0$,

\begin{eqnarray*}
Re\big(\mathcal{A}_{\alpha,\beta}U,U\big)_{\mathcal{H}}&\leq&
\bigg(\frac{\kappa}{2\varepsilon}-\frac{\xi}{2\tau}\bigg)\|z(1)\|_H^{2}+
\bigg(\frac{\varepsilon\kappa}{2}-a+\frac{\xi}{2\tau}\bigg)
\|B^{*}\theta\|_H^{2}.
\end{eqnarray*}

\medskip

As in the previous case, one has for  $\xi\in\mathcal{J}_{a,\kappa,\tau}$ and $\varepsilon=:\varepsilon_0 \varepsilon \in \left[\frac{\kappa\tau}{\xi},\frac{2}{\kappa}\big(a-\frac{\xi}{2\tau}\big)\right]$
\begin{equation}\label{s48}
Re\big(\mathcal{A}_{\alpha,\beta}U,U\big)_{\mathcal{H}}\leq 0,\quad
\forall \, U \in D(\mathcal{A}_{\alpha,\beta}).
\end{equation}

\end{proof}

We need the following result.

\begin{lemma}\rm\label{lemma23}
Assume that $a\geq \kappa$. Let
$\ds\xi\in\mathcal{J}_{a,\kappa,\tau}$. We have
\begin{equation}\label{s49}
\mathbb{C}_0
\subset\rho\big(\mathcal{A}_{\alpha,\beta}\big).
\end{equation}
\end{lemma}

\begin{proof}
Let $\lambda\in\mathbb{C}_0$ and
fix $(f,g,p,h)^{\top}\in\mathcal{H}$. Wee seek
a unique solution $U=(u,v,\theta,z)^{\top}\in D(\mathcal{A}_{\alpha,\beta})$ of
\begin{equation*}
\big(\lambda Id-\mathcal{A}_{\alpha,\beta}\big)
\left(
\begin{array}{c}
u \\
v \\
\theta \\
z \\
\end{array}
\right)=
\left(
\begin{array}{c}
f \\
g \\
p \\
h \\
\end{array}
\right),
\end{equation*}

that is verifying
\begin{equation}\label{s50}
\left\{
\begin{array}{ll}
\lambda u-v=f, &  \\
\ds\lambda v+A^{1/2}\big(A^{1/2}u-A^{-1/2}C\theta\big)=g, &  \\
\lambda\theta+\kappa B\big(z(1)+\frac{a}
{\kappa}B^{*}\theta\big)+ C^{*}v=p, &  \\
\ds\lambda z+\frac{1}{\tau}z_\rho=h. &
\end{array}
\right.
\end{equation}
If we have found $u$ with the appropriate regularity, then,
\begin{equation}\label{s51}
v=\lambda u-f.
\end{equation}
To determine z, we set $\ds z(0)=B^{*}u$. Then,
by $(\ref{s50})_{4}$, we get
\begin{equation}\label{s52}
\ds z(\rho)=e^{-\lambda\tau\rho}B^{*}\theta+\tau
e^{-\lambda\tau\rho}\int_{0}^{\rho}h(s)e^{\lambda\tau s}ds.
\end{equation}
In particular, we have
\begin{equation*}\label{s53}
z(1)=e^{-\lambda\tau}B^{*}\theta+z_{0}
\end{equation*}
with
\begin{equation*}\label{s54}
z_{0}=\tau e^{-\lambda\tau}\int_{0}^{1}h(s,.)
e^{\lambda\tau s}ds\in L^{2}\big((0,1),H\big).
\end{equation*}
Taking the inner product of $(\ref{s50})_{2}$
with $\ds x\in D(A^{\frac{1}{2}})$ in $H$ yields
\begin{equation}\label{s55}
\ds\lambda^{2}(u,x)_{H}+(A^{\frac{1}{2}}u,A^{\frac{1}{2}}x)_{H}-
(\theta,C^*x)_{H}=(g+\lambda f,x)_{H}, 
\end{equation}
for all $\ds x\in D(A^{\frac{1}{2}})$.

\medskip

Keeping in mind (\ref{s51}), the inner product of $(\ref{s50})_{3}$
with $\ds y\in D(B^{*})$ in $H$ yields
\begin{equation}\label{s56}
\ds\lambda(\theta,y)_{H}+(\kappa e^{-\lambda\tau}+a)
(B^{*}\theta,B^{*}y)_{H}+
\lambda(C^{*}u,y)_{H}=(p,y)_H+(C^*f,y)_{H}-\kappa( z_{0},B^{*}y)_H.
\end{equation}

 Next, summing (\ref{s56}) and (\ref{s55})
multiplied by $\ds\overline{\lambda}$,
we see that problems (\ref{s55}) and (\ref{s56}) can be formulated as
\begin{equation}\label{s58}
\Phi\big((u,\theta),(x,y)\big)=G(x,y)
\end{equation}
with
\begin{eqnarray*}
\Phi\big((u,\theta),(x,y)\big)&=&\lambda|\lambda|^{2}(u,x)_{H}+
\overline{\lambda}(A^{\frac{1}{2}}u,A^{\frac{1}{2}}x)_{H}-
\overline{\lambda}
(\theta,C^*x)_{H}
\\
&+&\lambda(\theta,y)_{H}+(\kappa e^{-\lambda\tau}+a)
(B^{*}\theta,B^{*}y)_{H}
+\lambda(C^{*}u,y)_{H},
\end{eqnarray*}
and
$$\ds G(x,y)=\overline{\lambda}(\lambda f+g,x)_{H}+(p,y)_H+(C^*f,y)_{H}-\kappa( z_{0},B^{*}y)_H.$$

It is obvious that the sesquilinear form $\Phi$ and the anti-linear form $G$ are continuous on $\mathcal{F}\times\mathcal{F}$ and  $\mathcal{F}$ respectively, where $\mathcal{F}$ is the Hilbert space defined by $\mathcal{F}:=D(A^{\frac{1}{2}})\times D(B^{*})$.

On the other hand, we have, for all $(x,y)\in\mathcal{F}$, 
\begin{eqnarray}
\ds \textrm{Re}\big(\Phi\big((x,y),(u,\theta)\big)\big)&=&
\textrm{Re}(\lambda)|\lambda|^{2}\|x\|_H^{2}+\textrm{Re}(\lambda)\|A^\frac{1}{2}x\|_H^{2}
+\textrm{Re}(\lambda)\|y\|_H^{2}\notag\\&+&
\big(\kappa  e^{-\tau\textrm{Re}(\lambda)}\cos(\textrm{Im}\lambda 
 \tau)+a\big)\|B^{*}y\|_H^{2}.\label{s59}
\end{eqnarray}

\noindent Since $a\geq \kappa$, it follows from (\ref{s59}) that $\Phi$ is coercive
on $\mathcal{F}\times\mathcal{F}$.
By the Lax-Milgram lemma, equation (\ref{s59})
has a unique solution $\ds(u,\theta)\in D(A^{\frac{1}{2}})\times D(B^{*})$.

\medskip

Now, we define $v$ by $v=\lambda u-f$, that belongs to $D(A^{1/2})$ and $z$ by (\ref{s52}), that belongs to $H^1((0,1),H)$. Next, by taking respectively $x=0$ and $y=0$ in (\ref{s58}), we deduce that that,
$$\ds\kappa z(1)+aB^{*}\theta
\in D(B)
\quad\textrm{ and }\quad \lambda\theta+\kappa B\big(z(1)+\frac{a}
{\kappa}B^{*}\theta\big)+ C^{*}v=p,$$
and
$$\ds A^{1/2}u-A^{-1/2}C\theta\in D(A^{1/2})\quad\textrm{ and }\quad
\lambda v+A^{1/2}\big(A^{1/2}u-A^{-1/2}C\theta\big)=g.$$
We conclude that $\lambda Id-\mathcal{A}_{\alpha,\beta}$ is bijective for
all $\lambda\in\mathbb{C}_0$ (in particular, $\mathcal{A}_{\alpha,\beta}$ is m-dissipative)

\medskip

Furthermore $\lambda \in \rho(\mathcal{A}_{\alpha,\beta})$, for every $\lambda \in\mathbb{C}_0$. 
\end{proof}

As a consequence of the m-dissipativness of  $\mathcal{A}_{\alpha,\beta}$ , we have the following result:

\begin{proposition}\label{p4.1}
Assume that $a\geq\kappa$ and let
$\ds\xi\in\mathcal{J}_{a,\kappa,\tau}$.
 Then, the system (\ref{s1}) is well-posed.
More precisely, the operator $\mathcal{A}_{\alpha,\beta}$ generates a
$\mathcal{C}_0$-semigroup of contractions on $\mathcal{H}$.
Thus, for every
$(u_0,v_0,\theta_0,z_0)\in \mathcal{H}$, there exists a unique solution
$(u,v,\theta,z)\in\mathcal{C}\big((0,+\infty),\mathcal{H}\big)$ of (\ref{s7'}).
In addition, if the initial datum
$(u_0,v_0,\theta_0,z_0)\in D(\mathcal{A}_{\alpha,\beta})$,
then the solution satisfies
$(u,u^\prime,\theta,z)\in\mathcal{C}\big((0,+\infty),D(\mathcal{A}_{\alpha,\beta})\big)\cap
\mathcal{C}^{1}\big((0,+\infty),\mathcal{H}\big)$.
\end{proposition}



\subsubsection{Exponential stability of the delayed coupled system (\ref{s1''})}

\begin{theorem}\rm\label{theorem41"}
Suppose that system (\ref{s1''}) satisfies $(\mathcal{H}_1)$ and $(\mathcal{H}_2)$, where

\medskip

$(\mathcal{H}_2)$:\;\;\;(i)\;\; $D(B^*)\subset D(A^{\frac{1}{2}-2\beta}C)$\;\; and there exists $c>0$ such that $$\|A^{\frac{1}{2}-2\beta}Cv\|_H\leq c\|B^*v\|_{H},\;\;\forall\
v\in D(B^{*}).$$ 
(ii)\;\; $A^{-2\beta}C$ is bounded and $B^{*}(A^{-2\beta}C)^*$ is (or can be extend to a) bounded operator.

  Assume that $a>\kappa$ and let
$\ds\xi \in\mathring{\mathcal{J}}_{a,\kappa,\tau}$.
Then, the semigroup $e^{t\mathcal{A}_{\alpha,\beta}}$  is exponentially stable.
\end{theorem}
\begin{remark} \label{re3.2}
    \begin{enumerate}
    \item For $\beta\geq \frac{1}{4}$, a sufficient condition for (i) (of $(\mathcal{H}_2)$)  to hold is: $D(B^*)\subset D(C)$\;\; and there exists $c>0$ such that $$\|Cv\|_H\leq c\|B^*v\|_{H},\;\;\forall\
v\in D(B^{*}).$$  
        \item Note that $A^{-2\beta}C=(C^*A^{-2\beta})^*$ (because $C$ is densely defined), then to prove that $A^{-2\beta}C$ is bounded, it suffices to prove that $C^*A^{-2\beta}$ is bounded (or equivalently, $D(A^{2\beta}) \subset D(C^*))$. 
    \end{enumerate}
\end{remark}
\begin{proof}[Proof of Theorem \ref{theorem41"}]
According to Lemma \ref{lemma23} the condition (\ref{s24})
is satisfied. 

\medskip

Now, we prove that condition (\ref{s25})
in Lemma \ref{lemma exp.stab} is satisfied. Suppose that (\ref{s25}) is not satisfied, then
there exists a sequence of complex numbers $\lambda_n$ such that
$\textrm{Re}(\lambda_n)\geq0$, $|\lambda_n|\rightarrow+\infty$
and a sequence
of unit vectors $U_n=(u_n,v_n,\theta_n,z_n)^\top\in D(\mathcal{A}_{\alpha,\beta})$
with
\begin{equation}\label{s60}
\underset{n\rightarrow+\infty}\lim\|\big(\lambda_n I-\mathcal{A}_{\beta,\alpha}\big)U_n \|_H=0,
\end{equation}

\noindent that is

\begin{equation}\label{s61}
\|A^{\frac{1}{2}}u_n\|_H^{2}+\|v_n\|_H^{2}+\|\theta_n\|_H^{2}+\xi\int_{0}^{1}\|z_n(\rho)\|_Hd\rho=1,
\quad\forall n\in\N
\end{equation}

and

\begin{equation}\label{s62}
\ds\lambda_n A^{\frac{1}{2}}u_n-A^{\frac{1}{2}}v_n\rightarrow0 \textrm{ in } H,
\end{equation}

\begin{equation}\label{s63}
\ds \lambda_n v_n+A^{1/2}\big(A^{1/2}u_n-A^{-1/2}C\theta_n\big)\rightarrow0 \textrm{ in } H,
\end{equation}

\begin{equation}\label{s64}
\ds\lambda_n\theta_n+ B
\big(\kappa z_n(1)+aB^{*}\theta_n\big)
+ C^{*}v_n\rightarrow0 \textrm{ in } H,
\end{equation}

\begin{equation}\label{s65}
\ds\lambda_nz_n+\frac{1}{\tau}\partial_{\rho}z_n\equiv
h_n\rightarrow0 \textrm{ in }L^2\big((0,1)H\big).
\end{equation}
As in the proof of Theorem \ref{theorem41}, we find that
\begin{equation}\label{s66}
z_n(1)\underset{n\rightarrow+\infty}\rightarrow0\quad \text{ and }\quad
B^{*}\theta_n\underset{n\rightarrow+\infty}\rightarrow0\textrm{ in }H,
\end{equation}
\begin{equation}\label{s67}
\theta_n\underset{n\rightarrow+\infty}\rightarrow0\textrm{ in }H.
\end{equation}
\begin{equation}
Re(\lambda_n)\underset{n\rightarrow+\infty}\rightarrow0,
\end{equation}
and 
Moreover
\begin{equation}\label{s83}
\ds\|A^{\frac{1}{2}}u_n\|_H^{2}-\|v_n\|_H^{2}
\underset{n\rightarrow+\infty}\rightarrow0.
\end{equation}
(Since $Re(\lambda_n)\rightarrow0$ and $|\lambda_n|\rightarrow\infty$, then without loss of generality we have supposed that that $\textrm{Im}(\lambda_n)\rightarrow \infty$.)
\begin{lemma}
\begin{equation}\label{s83""}
\|v_n\|_H^{2}
\underset{n\rightarrow+\infty}\rightarrow0\textrm{ in }H,
\end{equation}
\end{lemma}
\begin{proof}

Using $(\mathcal{H}_2)$  we have that $\|A^{\frac{1}{2}-2\beta}Cv\|\leq \|B^{*}v\|$

Taking the inner product of (\ref{s63}) with $A^{-2\beta}C\theta_n$, then using that $A^{\frac{1}{2}}u_n$ is bounded, we get
\begin{equation}\label{s83''}
\lambda_n(v_n,A^{-2\beta}C\theta_n)_H
\underset{n\rightarrow+\infty}\rightarrow0\textrm{ in }H.
\end{equation}
Taking the inner product of (\ref{s64}) with $(A^{-2\beta}C)^*v_n$, then using that $B^{*}\theta_n=o(1)$, $z_n(1)=o(1)$, $B^{*}(A^{-2\beta}C)^*$ is bounded, $v_n$ is bounded, and taking into account (\ref{s83''}), we deduce  
$$\|v_n\|_H^{2}
\underset{n\rightarrow+\infty}\rightarrow0\textrm{ in }H.$$

\end{proof}
By (\ref{s83}) we have also
$$\|A^{1/2}u_n\|_H^{2}
\underset{n\rightarrow+\infty}\rightarrow0\textrm{ in }H,$$
By integration of the identity (\ref{s65}) and using (\ref{s66}), we get
\begin{equation*}
\ds\|z_n\|_{L^{2}\big((0,1),H\big)}^{2}\rightarrow0.
\end{equation*}

\noindent All together, we have shown that $\{U_{n}\}$ converges to $0$
in $\mathcal{H}$. This clearly contradicts (\ref{s61}).
Thus, condition (\ref{s25}) holds and
Theorem \ref{theorem41"} is then proved.
\end{proof}
\subsubsection{A particular case}

 We focus on the second formulation (\ref{s1''}) by taking $C=C^*=A^\beta$ with $\beta \leq \frac{1}{2}$. then instead of $A^{-1/2}C$ we take $A^{\beta-\frac{1}{2}}$ which is bounded.
 
 Then, 
 $$\mathcal{A}_{\alpha,\beta}
\left(
\begin{array}{c}
u \\
v \\
\theta \\
z \\
\end{array}
\right)=
\left(
\begin{array}{c}
v \\
-A^{1/2}\big( A^{1/2}u-A^{\beta-\frac{1}{2}}\theta\big)  \\
\ds-\kappa B\big(z(1)+\frac{a}
{\kappa}B^*\theta\big)- A^{\beta}v \\
\ds-\frac{1}{\tau}z_{\rho} \\
  \end{array}
\right),$$
and
$$D(\mathcal{A}_{\alpha,\beta})=\left\{
\begin{array}{c}
(u,v,\theta,z)^{\top}\in D(A^{\frac{1}{2}})\times D(A^{\frac{1}{2}})
\times D(B^{*})\times H^{1}\big((0,1),H\big): \quad z(0)=B^*\theta, \\
\ds   A^{1/2}u-A^{\beta-\frac{1}{2}}\theta\in D(A^{1/2})  \quad\textrm{ and }\quad z(1)+\frac{a}
{\kappa}B^*\theta\in D(B)
\end{array}
\right\}.$$ 

The result in Proposition \ref{p4.1} is preserved under the condition
$(\mathcal{H}_1^\prime)$:\;\;\;  There exists $c_1 > 0$ such that $$\|B^{*}v\|_H\geq c_1\|v\|_{H},\;\;\forall\
v\in D(B^{*}).$$
Moreover, under the condition

\medskip

$(\mathcal{H}_2^\prime):$\;\;\;(i)\;\;
$D(B^*)\subset D(A^{\frac{1}{2}-\beta})$\;\; and there exists $c>0$ such that $$\|(A^{\frac{1}{2}-\beta})v\|_H\leq c\|B^*v\|_{H},\;\;\forall\
v\in D(B^{*}),$$

(ii)\;\;  $B^{*}A^{-\beta}$ bounded, equivalently, $D(A^\beta)\subset D(B^*)$,

\medskip

 the result in Theorem \ref{theorem41"} is preserved. 
 
 Note here that property $(\mathcal{H}_1^\prime)$ is a consequence of property (i) in $(\mathcal{H}_2^\prime)$.  Moreover,
 we have the following result
\begin{theorem}
Taking $C=C^*=A^\beta$ with $\beta \leq \frac{1}{2}$ and so $\mathcal{A}_{\alpha,\beta}$ is defined as above. Suppose that $D(B^*)= D(A^{\alpha/2})$\;\; and $\|(A^{\alpha/2})v\|_H\leq c\|B^*v\|_{H},\;\;\forall\
v\in D(B^{*})=D(A^{\alpha/2}).$ Then  the semigroup $e^{t\mathcal{A}_{\alpha,\beta}}$ has the following stability properties:

(i)  In $S $, it is exponentially  stable.

(ii) In $S_1 $, it is polynomially stable of order $\frac{1}{2(\alpha-2\beta)}$.

(iii) In $S_2$, it is polynomially stable of order $\frac{1}{2-2(\alpha+2\beta)}$. 
\end{theorem}
\begin{proof}
The justification for (i) is  explained above. For (ii) and (iii), it suffices to repeat the proof of Theorem \ref{theorem41'} with a slight modification.
\end{proof}
\subsection{Applications} 
\subsubsection{A thermoelastic beam 1}
We consider the one dimensional delayed thermoelastic  beam model in $(0,L)$, $L>0$:

\begin{equation}\label{exp33}
\left\{
\begin{array}{ll}
u_{tt}(t,x)+u_{xxxx}(x,t)-\theta_{xx}(x,t)=0, &
\quad (x,t)\in(0,L)\times(0,+\infty), \\
\theta_{t}(x,t)-\kappa\theta_{xx}(x,t-\tau)-a\theta_{xx}(x,t)+u_{txx}(x,t)=0, &
\quad(x,t)\in(0,L)\times(0,+\infty), \\
u(0,t)=u(L,t)=u_{xx}(0,t)=u_{xx}(L,t)=0, & \quad t\in(0,+\infty), \\
u(x,0)=u^0(x),u_t(x,0)=u^1(x), & \quad t\in(0,\tau), \\
\theta(0,t)=\theta(L,t)=0, & \quad t\in(0,+\infty), \\
\theta(x,0)=\theta^0(x), &  \\
\theta_x(x,t)=g_0(x,t), &\quad-\tau\leq t<0, x\in (0,L),
\end{array}
\right.
\end{equation}
where $L$ and $a$ are real positive constants.

\medskip

It is a concrete example of the last related system, with $H=L^2(0,L)$, $A=\frac{\partial^4}{\partial x^4}:D(A)=H_0^2(0,L)\cap H^4(0,L)\rightarrow L^2(0,L)$, $\alpha=\beta=\frac{1}{2}$,   $A^\alpha=A^\beta=A^{1/2}=-\frac{\partial^2}{\partial x^2}:D(A^{1/2})=H_0^1(0,L)\cap H^2(0,L)\rightarrow L^2(0,L)$,  $B=-\frac{\partial}{\partial x}:D(B)=H^1(0,L)\rightarrow L^2(0,L)$, $B^*=\frac{\partial}{\partial x}:D(B^*)=H_0^1(0,L)\rightarrow L^2(0,L)$. We have that $A^{\alpha}=A^{1/2}=BB^*$, $D(A^{1/4})=D(B^*)=H_0^1(0,L)$ and $\|(A^{1/4})v\|_H=\|B^*v\|_{H},\;\;\forall\
v\in D(B^{*})=H_0^1(0,L)$.

\medskip

The operator $\mathcal{A}_{\frac{1}{2},\frac{1}{2}}$ is given by
\begin{equation*}
\ds\mathcal{A}_{\frac{1}{2},\frac{1}{2}}
\left(
\begin{array}{c}
u \\
v \\
\theta \\
z \\
\end{array}
\right)=
\left(
\begin{array}{c}
v \\
\ds\left( -u_{xx}+\theta\right)_{xx}  \\
\kappa\big(z(\cdot,1)+\frac{a}{\kappa}\theta_{x}\big)_x-v_{xx} \\
\ds-\frac{1}{\tau}z_{\rho} \\
  \end{array}
\right),
\end{equation*}
with domain
\begin{equation*}
D(\mathcal{A}_{\frac{1}{2},\frac{1}{2}})=\left\{
\begin{array}{c}
(u,v,\theta,z)^{\top}\in \left( H_{0}^{1}\big(0,L\big) \cap H^{2}\big(0,L\big)\right) \times \left( H_{0}^{1}\big(0,L\big)\cap H^{2}\big(0,L\big)\right)
\times  H_{0}^{1}\big(0,L\big) \times H^{1}\big((0,1),L^{2}(0,L)\big): \\
\ds \quad z(\cdot,0)=\theta_x,\quad -u_{xx}+\theta\in H_{0}^{1}\big(0,L\big)\cap H^2\big(0,L\big) \quad\textrm{ and }\quad z(\cdot,1)+\frac{a}{\kappa}\theta_x\in H^{1}\big(0,L\big)
\end{array}
\right\},
\end{equation*}
in the Hilbert space
\begin{equation*}
\mathcal{H}= \left(H_{0}^{1}\big(0,L\big)\cap H^{2}\big(0,L\big)\right) \times
L^{2}\big(0,L\big) \times L^{2}\big(0,L\big)\times
L^{2}\big((0,L)\times (0,1)\big),
\end{equation*}



Then, the Cauchy problem 
\begin{equation*}
\left\{
\begin{array}{ll}
U^\prime (t) =\mathcal{A}_{\frac{1}{2},\frac{1}{2}}U(t), \, t > 0, & \\
\ds U(0)=\big(u_{0},u_{1},\theta_{0},g_0(\cdot, -\tau \cdot)\big)^{\top},
\end{array}
\right.
\end{equation*}
is well-posed in the Hilbert space $\mathcal{H}$.

\begin{theorem}
If $a> \kappa$ and $\xi \in \mathring{\mathcal{J}}_{a,\kappa,\tau}$, the system (\ref{exp33}) is exponentially stable.
\end{theorem}
\subsubsection{A thermoelastic beam 2}

We consider the one dimentional delayed thermoelastic  beam model in $(0,L)$, $L>0$:

\begin{equation}\label{exp22}
\left\{
\begin{array}{ll}
u_{tt}(t,x)+u_{xxxx}(x,t)-\theta(x,t)=0, &
\quad (x,t)\in(0,L)\times(0,+\infty), \\
\theta_{t}(x,t)-\kappa\theta_{xx}(x,t-\tau)-a\theta_{xx}(x,t)+u_{t}(x,t)=0, &
\quad(x,t)\in(0,L)\times(0,+\infty), \\
u(0,t)=u(L,t)=u_{xx}(0,t)=u_{xx}(L,t)=0, & \quad t\in(0,+\infty), \\
u(x,0)=u^0(x),u_t(x,0)=u^1(x), & \quad t\in(0,\tau), \\
\theta(0,t)=\theta(L,t)=0, & \quad t\in(0,+\infty), \\
\theta(x,0)=\theta^0(x), &  \\
\theta_x(x,t)=g_0(x,t), &\quad-\tau\leq t<0, x\in (0,L),
\end{array}
\right.
\end{equation}
where $L$ and $a$ are real positive constants.

It is a second concrete example of the last related system, with $\alpha=\frac{1}{2}$, $\beta=0$, $H=L^2(0,L)$, $A=\frac{\partial^4}{\partial x^4}:D(A)=H_0^2(0,L)\cap H^4(0,L)\rightarrow L^2(0,L)$, then  $A^\alpha=A^{1/2}=-\frac{\partial^2}{\partial x^2}:D(A^{1/2})=H_0^1(0,L)\cap H^2(0,L)\rightarrow L^2(0,L)$, $B=-\frac{\partial}{\partial x}:D(B)=H^1(0,L)\rightarrow L^2(0,L)$, $B^*=\frac{\partial}{\partial x}:D(B^*)=H_0^1(0,L)\rightarrow L^2(0,L)$. We have that $A^{1/2}=BB^*$, $D(A^{1/4})=D(B^*)=H_0^1(0,L)$ and $\|(A^{1/4})v\|_H=\|B^*v\|_{H},\;\;\forall\
v\in D(B^{*})=H_0^1(0,L)$.  

The operator $\mathcal{A}_{\frac{1}{2},0}$ is given by
\begin{equation*}
\ds\mathcal{A}_{\frac{1}{2},0}
\left(
\begin{array}{c}
u \\
v \\
\theta \\
z \\
\end{array}
\right)=
\left(
\begin{array}{c}
v \\
\ds-u_{xxxx}+\theta \\
\kappa\big(z(\cdot,1)+\frac{a}{\kappa}\theta_{x}\big)_x-v \\
\ds-\frac{1}{\tau}z_{\rho} \\
  \end{array}
\right),
\end{equation*}
with domain
\begin{equation*}
D(\mathcal{A}_{0,\frac{1}{2}})=\left\{
\begin{array}{c}
(u,v,\theta,z)^{\top}\in \left( H_{0}^{1}\big(0,L\big) \cap H^{2}\big(0,L\big)\right) \times \left( H_{0}^{1}\big(0,L\big)\cap H^{2}\big(0,L\big)\right)
\times  H_{0}^{1}\big(0,L\big) \times L^{2}\big((0,L),H^{1}(0,1)\big): \\
\ds  u_{xx}\in H_{0}^{1}\big(0,L\big) \cap H^{2}\big(0,L\big),\quad z(\cdot,0)=\theta_x\quad\textrm{ and }\quad z(\cdot,1)+\frac{a}{\kappa}\theta_x\in H^{1}\big(0,L\big)
\end{array}
\right\},
\end{equation*}
in the Hilbert space
\begin{equation*}
\mathcal{H}= \left(H_{0}^{1}\big(0,L\big)\cap H^{2}\big(0,L\big)\right) \times
L^{2}\big(0,L\big)\times L^{2}\big(0,L\big)\times
L^{2}\big((0,L)\times (0,1)\big),
\end{equation*}



One has the following polynomial stability result.

\begin{theorem}
If $a> \kappa$ and $ \xi \in \mathring{\mathcal{J}}_{a,\kappa,\tau}$, the system (\ref{exp22}) is polynomially stable of order $1$.
\end{theorem} 
\subsubsection{A thermoelastic string}
We consider the following system with delay
\begin{equation}\label{exp44}
\left\{
\begin{array}{ll}
u_{tt}(t,x)-u_{xx}(x,t)+\theta_{x}(x,t)=0, &
\quad (x,t)\in(0,L)\times(0,+\infty), \\
\theta_{t}(x,t)-\kappa\theta_{xx}(x,t-\tau)-a\theta_{xx}(x,t)+u_{xt}(x,t)=0, &
\quad(x,t)\in(0,L)\times(0,+\infty), \\
u(0,t)=u(L,t)=0, &  \\
u(x,0)=u^0(x),u_t(x,0)=u^1(x), & \quad t\in(0,\tau), \\
\theta(0,t)=\theta(L,t)=0, & \quad t\in(0,+\infty), \\
\theta(x,0)=\theta^0(x), &  \\
\theta_x(x,t)=g_0(x,t), &\quad-\tau\leq t<0, x\in (0,L),
\end{array}
\right.
\end{equation}
where $L$ and $a$ are real positive constants.

\medskip

We thus find a system considered in \cite{MuKa13}, by considering here Dirichlet conditions instead of Neumann conditions for $\theta$. It is a direct application of the initial related system: 
$H=L^{2}\big(0,L\big)$, $A=-\frac{d^2}{dx^{2}}:
 D(A)=H_0^1(0,L)\cap H^2(0,L)\rightarrow L^2(0,L)$, $\alpha=1$, $\beta=\frac{1}{2}$,  $B=C=\frac{\partial}{\partial x}:D(B)=H^1(0,L)\rightarrow L^2(0,L)$, $B^*=C^*=-\frac{\partial}{\partial x}:D(B^*)=H_0^1(0,L)\rightarrow L^2(0,L)$ and $A=BB^*$. 

Now, $D(A^{1/2})=H^1_0(0,L)=D(C^*)$, $D(B^*)=H^1_0(0,L)\subset D(C)=H^1(0,L)$ and $\|Cv\|_H= \|B^*v\|_{H},\;\;\forall v\in D(B^{*})$. The condition $(\mathcal{H}_1)$ is then satisfied, so system (\ref{exp44}) is well-posed.  

On the other hand, in view of Remark \ref{re3.2}, assumption (i) and the first part of assumption (ii) in $\mathcal{H}_2$ are satisfied.

The second part of assumption (ii) in $\mathcal{H}_2$ is satisfied too, because  for every $v\in \mathcal{C}_0^\infty$, we have $\|A^{-1}CBv\|_H=\|v\|_H$. Hence,
\begin{theorem}
If $a> \kappa$ and $\xi \in \mathring{\mathcal{J}}_{a,\kappa,\tau}$, then the system (\ref{exp44}) is exponentially stable.
\end{theorem}

\section*{Conflict of interest}
The authors declare that they have no conflict of interest.

\section*{Availability of data}
The authors declare that all data in this paper is available.


\begin{thebibliography}{99}

\bibitem{ADB93}
C. Abdallah, P. Dorato, J. Benitez-Read and R. Byrne, Delayed positive feedback can stabilize oscillatory system, ACC San Fransisco (1993), 3016--3107.

\bibitem{Benhassi}
E. M. Ait Benhassi, K. Ammari, S. Boulite and L. Maniar, Feedback stabilization of a class of evolution
equation with delay, {\em J. Evol. Equations}, {\bf 9} (2009), 103--121.

\bibitem{ABB99} F. Ammar-Khodja, A. Bader and A. Benabdallah, Dynamic stabilization of systems via decoupling techniques, {\em ESAIM Control Optim. Calc. Var.}, {\bf 4} (1999), 577--593.

\bibitem{AmBe00} F. Ammar-Khodja and A. Benabdallah, Sufficient conditions for uniform stabilization of second order equations by dynamical controllers, {\em Dynamics of Continuous Discrete and Impulsive systems}, {\bf 7} (2000), 207--222.

\bibitem{AHbook}
 K. Ammari and F. Hassine, {\em Stabilization of Kelvin-Voigt damped systems},
Adv. Mech. Math., 47
Birkhäuser/Springer, Cham, 2022.

\bibitem{ASbook} 
 K. Ammari and F. Shel, {\em Stability of elastic multi-link structures},
SpringerBriefs Math.
Springer, Cham, 2022.

\bibitem{ALS23}
K. Ammari, Z. Liu, and F. Shel, Note on stability of an abstract coupled hyperbolic-parabolic systel; singular case, {\em Appl. Math. Lett.,} {\bf 141} (2023), Paper No. 108599, 7 pp.

\bibitem{Ammari}
K. Ammari, S. Nicaise, and C. Pignotti, Stability of
abstract-wave equation with delay and a Kelvin-Voigt
damping, {\em Asymptot. Anal.}, {\bf 95} (2015), 21--38.

\bibitem{Ammari1}
K. Ammari and S. Nicaise, {\em Stabilization of elastic systems by collocated feedback,} Lecture notes in 
Mathematics, vol. 2124, Springer, Cham, 2015.

\bibitem{Ammari2}
K. Ammari, S. Nicaise and C. Pignotti, Stabilization by switching time-delay,
{\em Asymptotic Analysis}, {\bf 83} (2013), 263--283.

\bibitem{AvLa97}
G. Avalos and I. Lasiecka, Exponential stability of a thermoelastic system without mechanical dissipation.
{\em Rend. Instit. Mat. Univ. Trieste Suppl.}, {\bf 28} (1997), 1--28.

\bibitem{BattyDuck} C. J. K. Batty and T. Duyckaerts, Non-uniform stability for bounded semi-groups on {B}anach spaces, {\em J. Evol. Equ.}, {\bf 8} (2008), 765--780.

\bibitem{Batkai}
A. B\'{a}tkai and S. Piazzera, {\em Semigroups for Delay Equations},
Research Notes in Mathematics,
10, A.K. Peters, Wellesley MA, 2005.

\bibitem{BoTo10}
A.~Borichev and Y.~Tomilov, \emph{Optimal polynomial decay of functions and
  operator semigroups}, Math. Ann., \textbf{347} (2010), 455--478.

\bibitem{Dat88}
R. Datko, Not all feedback stabilized hyperbolic systems are robust with respect to
small time delays in their feedbacks, {\em SIAM J. Control Optim.}, {\bf 26} (1988), 697--713.


\bibitem{Dat97}
R. Datko, Two examples of ill-posedness with respect to time delays revisited,
{\em IEEE Trans. Automatic Control.}, {\bf 42} (1997), 511--515.

\bibitem{DLP86}
R. Datko, J. Lagnese and M. P. Polis, An example of the effect of time delays in boundary feedback stabilization of wave equations,
{\em SIAM J. Control Optim.}, {\bf 24} (1986), 152--156.

\bibitem{DS}
D. S. Chandrasekharaiah, Hyperbolic thermoelasticity: A review of recent literature, {\em Appl. Mech. Rev.}, {\bf 51} (1998), 705--729.

\bibitem{Dreher}
M. Dreher, R. Quintanilla and R. Racke, Ill-posed problems in thermomechanics, {\em Appl. Math. Letters.}, {\bf 22} (2009), 
1374--1379.

\bibitem{Gea78}
L.~M. Gearhart, \emph{Spectral theory for contraction semigroups on {H}ilbert
  space}, Trans. Amer. Math. Soc., \textbf{236} (1978), 385--394.

\bibitem{GRT92}
J.S. Gibson, I.G. Rosen and G. Tao, Approximation in control thermùoelastic systems, {\em SIAM J. Control Optim.}, {\bf 30} (1992), 1163--1189.

\bibitem{Gugat}
M. Gugat, Boundary feedback stabilization by time delay for one-dimensional wave
equations, {\em IMA J. Math. Control Inform.}, {\bf 27} (2010), 189--203.

\bibitem{HaLi13} J. Hao and Z. Liu, Stability of an abstract system of coupled hyperbolic and parabolic equations, {\em Z. Angew. Math. Phys.}, {\bf 64} (2013), 1145--1159.

\bibitem{HLY15} J. Hao, Z. Liu and J. Yong, Regularity analysis for an abstract system of coupled hyperbolic and parabolic equations, {\em Journal of Differential Equations}, {\bf 259} (2015), 4763--4798.  

\bibitem{Zwart}
B. Jacob and H. Zwart, {\em Linear Port-Hamiltonian Systems on Infinite-dimensional Spaces}, Operator Theory: Advances and Applications, 223,
Birkh\"auser, 2012.

\bibitem{Kim92}
J. U. Kim, On the energy decay of a linear thermoelastic bar and plate. {\em SIAM J. Math. Anal.}, {\bf 23} (1992), 
889--899. 

\bibitem{KhSh21}
 S. M. Khatir and F. Shel, \emph{Well-posedness and exponential stability of a
  thermoelastic system with internal delay}, {\em Applicable Analysis}, {\bf 101} (2022), 4851--4865.
  
\bibitem{LiLi97}
K. Liu and Z. Liu, Exponential stability and analyticity of abstract linear thermoelastic systems, {\em Z. angew. Math. Phys.}, {\bf 48} (1997), 885--904.  

\bibitem{LiZh93}
Z. Liu and S. Zheng, Exponential stability of semigroup associated with thermoelastic system, {\em Quart. Appl.
Math.}, {\bf 51} (1993), 535--545.

\bibitem{LiZh97}
Z. Liu and S. Zheng, Exponential stability of the Kirchhoff plate with thermal or viscoelastic damping,
{\em Q. Appl. Math.}, {\bf 53} (1997), 551-564 

\bibitem{LiZh99}
Z. Liu and S. Zheng, \emph{Semigroups associated with dissipative systems},
  Chapman $\&$ Hall/CRC, 1999.

\bibitem{Messaoudi}
S. A. Messaoudi, A. Fareh and N. Doudi, Well posedness and exponential stability in a wave equation with
a strong damping and a strong delay, {\em J. Math. Phys.}, {\bf 57} (2016), 111501, 13 pp.

\bibitem{MuKa13}
M. I. Mustapha and M. Kafini, Exponential decay in thermoelastic systems with internal distributed delay, {\em Palest. J. Math.}, {\bf 2} (2013), 287--299.

\bibitem{Naf20}
S. Nafiri, On the impact of boundary conditions on weakly coupled thermoelastic wave model, submitted,
arXiv: 2010.03612.

\bibitem{Naf23}
S. Nafiri, Uniform polynomial decay and approximation in control of a family of abstract thermoelastic models, {\em J. Dyn. Control.
Syst.}, {\bf 29} (2023), 209--227.

\bibitem{NiPi06}
S. Nicaise and C. Pignotti, Stability and instability results of the wave equation with a delay term in the boundary or internal feedbacks, {\em SIAM J. Control Optim.}, {\bf 45} (2006), 
1561--1585.

\bibitem{NiPi08}
S. Nicaise and C. Pignotti, Stabilization of
the wave equation with
boundary or internal delay, {\em Int. Differ. Equations.}, {\bf 21} (2008), 935--958.

\bibitem{NPV11}
S. Nicaise, C. Pignotti, and J. Valein, Exponential instability results of
the wave equation with boundary time-varying delay, {\em Discrete Contin. Dyn. Syst. Ser. S.}, {\bf 4} (2011), 693--722.

\bibitem{NiPi18}
S. Nicaise and C. Pignotti, Well-posedness and stability results for
nonlinear abstract evolution equations with time delays, {\em J. Evol. Equ.}, {\bf 18} (2018), 
947--971.

\bibitem{Pruss}
J. Pr\"{u}ss, {\em Evolutionary Integral Equations and Applications}, Monographs Math., 87,
Birkh\"{a}user, Basel, 1993.


\bibitem{Jordan}
P. M. Jordan, W. Dai and R. E. Mickens, A note on the delayed heat equation:
Instability with respect to initial data, {\em Mech. Res. Comm.}, {\bf 35} (2008), 414--420.

\bibitem{Rac02}
R. Racke,  Thermoelasticity with second sound: exponential stability in linear and nonlinear 1-d, {\em Math. Meth. App. Sci.}, {\bf 25} (2002), 409--441.

\bibitem{Rac12}
R. Racke, Instability of coupled systems with delay, {\em Commun. Pure Appl. Anal.},
{\bf 11} (2012), 1753--1773.

\bibitem{RiRa95}
J. E. M. Rivera and R. Racke, Smoothing properties, decay and global existence
of solutions to nonlinear coupled systems of thermoelastic type, {\em SIAM J. Math. Anal.}, {\bf 26} (1995), 1547-1563.

\bibitem{RiRa96}
J. E. M. Rivera and R. Racke, Large solutions and smoothing properties for nonlinear thermoelastic
systems, {\em J. Differ. Equ.}, {\bf 127} (1996), 454--483.

\bibitem{SLR19}
H. D. F. Sare, Z. Liu and R. Racke, Stability of abstract thermoelastic systems with internal terms, {\em J. Differ. Equ.}, {\bf 267} (2019), 
7085--7134. 

\bibitem{SuBi80}
I. H. Suh and Z. Bien, Use of time delay action in the controller design, {\em IEEE Trans. Autom. Control.},
{\bf 25} (1980), 600--603.

\bibitem{Tucsnak}
M. Tucsnak and G. Weiss, {\em Observation and Control for Operator Semigroups,}
Birkh\"{a}user, Basel, Boston, Berlin, 2009.

\end{thebibliography}
\end{document}